\documentclass[11pt]{amsart}

\usepackage{a4}
\usepackage{amsfonts}
\usepackage{amsmath}
\usepackage{amsthm}
\usepackage{url}
\usepackage[all]{xy}
\usepackage{graphicx}
\usepackage{latexsym}
\usepackage{amssymb}
\usepackage[cp850]{inputenc}
\usepackage{epsfig}
\usepackage{psfrag}
\usepackage{amsfonts}
\usepackage{dsfont}

\catcode`\@=11

\long\def\@savemarbox#1#2{\global\setbox#1\vtop{\hsize\marginparwidth 
  \@parboxrestore\tiny\raggedright #2}}
\newcommand\lref[1]{\ref{#1}%
\@ifundefined{r@DisplaY #1}{}{ (#1)}}% Prints label as well as
\newcommand\fakelabel[2]{\@bsphack\if@filesw {\let\thepage\relax
   \newcommand\protect{\noexpand\noexpand\noexpand}%
\xdef\@gtempa{\write\@auxout{\string
      \newlabel{#1}{{#2}{\thepage}}}}}\@gtempa
   \if@nobreak \ifvmode\nobreak\fi\fi\fi\@esphack}

\def\SL@margintext#1{{\showlabelsetlabel{\tiny\{\SL@prlabelname{#1}\}}}}
\catcode`\@=12

\newtheorem{sat}{Theorem}[section]		\newtheorem{lem}[sat]{Lemma}
\newtheorem{kor}[sat]{Corollary}			\newtheorem{prop}[sat]{Proposition}

\newtheorem*{defi*}{Definition}			\newtheorem*{bei*}{Example}
\newtheorem*{sat*}{Theorem}				\newtheorem*{kor*}{Corollary}
\newtheorem*{rmk*}{Remark}				\newtheorem*{lem*}{Lemma}

\let\ssection=\section
\renewcommand{\section}{\setcounter{equation}{0}\ssection}

\newtheorem*{namedtheorem}{\theoremname}
\newcommand{\theoremname}{testing}
\newenvironment{named}[1]{\renewcommand{\theoremname}{#1}\begin{namedtheorem}}{\end{namedtheorem}}

\theoremstyle{remark}
\newtheorem*{bem}{Remark}

\newcommand{\BC}{\mathbb C}			\newcommand{\BH}{\mathbb H}
\newcommand{\BR}{\mathbb R}			
			
\newcommand{\BS}{\mathbb S}			\newcommand{\BZ}{\mathbb Z}

		\newcommand{\CF}{\mathcal F}
		\newcommand{\CH}{\mathcal H}

		\newcommand{\CN}{\mathcal N}
		
\newcommand{\CQ}{\mathcal Q}

\newcommand{\D}{\partial}
\newcommand{\DD}{\nabla}

\DeclareMathOperator{\Id}{Id}		%	Identit\"at
\DeclareMathOperator{\vol}{vol}		%	Volumen

\DeclareMathOperator{\inj}{inj}
\DeclareMathOperator{\diam}{diam}

\newcommand{\comment}[1]{}

\DeclareMathOperator{\Mod}{Mod}
\DeclareMathOperator{\ds}{ds}

\DeclareMathOperator{\loc}{loc}

\begin{document}

\title[]{Distributional limits of Riemannian manifolds and graphs with sublinear genus growth}
\author{Hossein Namazi, Pekka Pankka and Juan Souto}
\thanks{Pekka Pankka has been partially supported by project 256228 of the Academy of Finland. Juan Souto has been partially supported by NSERC Discovery and Accelerator Supplement grants.}
\begin{abstract}
In \cite{BS} Benjamini and Schramm introduced the notion of {\em distributional limit} of a sequence of graphs with uniformly bounded valence and studied such limits in the case that the involved graphs are planar. We investigate distributional limits of sequences of Riemannian manifolds with bounded curvature which satisfy a quasi-conformal condition. We then apply our results to somewhat improve Benjamini's and Schramm's original result on the recurrence of the simple random walk on limits of planar graphs. For instance, as an application we give a proof of the fact that for graphs in an expander family, the genus of each graph is bounded from below by a linear function of the number of vertices.
\end{abstract}

\maketitle

\section{Introduction}
For $K\ge 1$ and a closed Riemannian manifold $M$, let $\CQ(M,K)$ be the set of all Riemannian manifolds $M'$ with pinched sectional curvature $\vert\kappa_{M'}\vert\le 1$ such that there is a $K$-quasi-conformal homeomorphism $$f:M'\to M.$$
Suppose that $(M_i)$ is a sequence in $\CQ(M,K)$ and let $\CH$ be the space of all isometry classes of pointed metric spaces endowed with the pointed Gromov-Hausdorf topology. For each $i$ consider the Lebesgue measure $\vol_{M_i}$ on $M_i$ induced by the Riemannian metric, and let $\lambda_{M_i}$ be the push-forward of the probability measure $\frac 1{\vol_{M_i}(M_i)}\vol_{M_i}$ by the continuous map
$$\iota_{M_i}:M_i\to\CH,\ \ \iota_{M_i}(p)=(M_i,p).$$
It follows from our curvature assumption that, up to passing to a subsequence, the sequence of measures $\lambda_{M_i}$ converges in the weak-*-topology to a probability measure on $\CH$. If the sequence $(\lambda_{M_i})$ actually converges, then the limiting measure $\lambda$ is the {\em distributional limit} of the sequence $(M_i)$.

\begin{sat}\label{sat1}
Given $K\ge 1$ and a closed Riemannian manifold $M$ of dimension $d\ge 3$, let $(M_i)\subset\CQ(M,K)$ be a sequence with distributional limit $\lambda$. If $\vol(M_i)\to\infty$, then the set of those $(X,x)\in\CH$ such that $X$ is a Riemannian manifold $K$-quasi-conformally equivalent to $\BR^d$ or $\BR^d\setminus\{0\}$ has full $\lambda$-measure. 
\end{sat}

We can reformulate the content of Theorem \ref{sat1} by saying that for a sequence $(M_i)\subset\CQ(M,K)$ with $\vol(M_i)\to\infty$, and for randomly chosen base points $p_i\in M_i$, the possible Gromov-Hausdorff limits of subsequences of $(M_i,p_i)$ are quasi-conformally equivalent to $\BR^d$ and $\BR^d\setminus\{0\}$. To appreciate this statement, notice that for each $d\ge 2$ there are sequences $(M_i)\subset\CQ(\BS^d,1)$, and $(p_i)$ with $p_i\in M_i$ such that $(M_i,p_i)$ converges to $(X,x)$ where $X$ has infinite topology (see section \ref{sec:weak}). The statement of Theorem \ref{sat1} is that the choice of this sequence of base points has been extremely biassed. 

Notice also that the condition that the manifolds $(M_i)$ in the sequence are uniformly $K$-quasi-conformal to a fixed manifold is necessary for the theorem to hold. It is easy to use the Whitehead manifold to construct a sequence $(M_i)$ of manifolds diffeomorphic to $\BS^3$, with sectional curvature bounded in absolute value by one, injectivity radius bounded from below, volume growing without bounds, and with distributional limit supported by the set of manifolds with infinite topology (see section \ref{sec:weak}).
\medskip

As always, knowing something about the Gromov-Hausdorff limit of a sequence of Riemannian manifold translates into knowledge about the members of the sequence. For instance, we derive from Theorem \ref{sat1}:

\begin{sat}\label{cheeger}
Fix $K\ge 1$ and a closed Riemannian manifold $M$ of dimension $d\ge 3$. If $(M_i)\subset\CQ(M,K)$ is a sequence such that $\vol(M_i)\to\infty$, then 
$$\lim_{i\to\infty}h(M_i)= 0$$
where $h(M_i)$ is the Cheeger constant of $M_i$.
\end{sat}

Recall that the Cheeger constant of a possibly non-compact Riemannian manifold $N$ is 
$$h(N)=\inf\frac{\mathrm{Area}(\D A)}{\min\{\vol(A),\vol(B)\}}$$
where the infimum is taken over all decompositions $M=A\cup B$ where $A$ and $B$ have disjoint interior and $\vol(A)<\infty$. It follows from Theorem \ref{cheeger} and the work of Buser \cite{Buser} that for sequences $(M_i)\subset\CQ(M,K)$ with $\vol(M_i)\to\infty$ we also have $\lambda_1(M_i)\to 0$ where $\lambda_1(\cdot)$ is the first eigenvalue of the Laplacian.
\medskip

Returning to the setting of Theorem \ref{sat1}, recall the assumption that the involved manifolds have at least dimension $3$. In fact, Theorem \ref{sat1} fails in dimension 2. For instance, there are sequences $(M_i)$ of Riemannian surfaces conformally equivalent to $\BS^2$ whose distributional limit $\lambda$ satisfies
$$\lambda\left(\{(X,x)\in\CH\vert\ X\ \hbox{has infinite topological type}\}\right)>0.$$
We discuss this example in section \ref{sec:2dim}.

The main difference between the 2-dimensional and the higher dimensional case is that in the latter local collapse cannot occur. More concretely, we use the Gromov-Zorich {\em global homeomorphism theorem} \cite{Gromov-book,Z2} to prove:

\begin{sat}\label{no-collapse}
Let $M$ be a closed Riemannian manifold of dimension $d\ge 3$. For every $K\ge 1$ there is $\epsilon$ positive such that every $M'\in\CQ(M,K)$ with diameter $\diam(M')\ge 1$ has injectivity radius $\inj(M')\ge\epsilon$.
\end{sat}

Basically, the $\diam(M')\ge 1$ assumption in Theorem \ref{no-collapse} rules out that $M$ and $M'$ are homothetic flat tori with $\vol_{M'}(M')$ very small. In fact, the condition $\diam(M')\ge 1$ is superfluous if there is no flat manifold which is $K$-quasi-conformally equivalent to $M$. Notice also that Theorem \ref{no-collapse} implies that the subset of $\CQ(M,K)$ consisting of manifolds with volume at most $V$ is compact. This justifies the assumption in Theorem \ref{sat1} and Theorem \ref{cheeger} that the volume of the involved manifolds grows without bounds.
\medskip

As we mentioned above, the failure of Theorem \ref{sat1} for surfaces is related to the lack of uniform bounds for the injectivity radius. In fact, Theorem \ref{sat1} holds in the 2-dimensional situation as long as we assume such a uniform bound. Moreover, in that situation we can use the uniformization theorem to prove a version of Theorem \ref{sat1} for sequences of surfaces with slow genus growth:

\begin{sat}\label{sat12}
Suppose that $(M_i)$ is a sequence of closed Riemannian surfaces with 
$$\vert\kappa_{M_i}\vert\le 1\ \hbox{and}\ \inj(M_i)>\epsilon>0$$ 
for all $i$. Suppose also that $(M_i)$ has distributional limit $\lambda$ and that 
$$\lim_{i\to\infty}\frac{g(M_i)+1}{\vol_{M_i}(M_i)}= 0$$
where $g(M_i)$ is the genus of $M_i$. Then $\lambda$ is supported by the set of Riemannian surfaces conformally equivalent to $\BC$ or $\BC^*$.
\end{sat}

Armed with Theorem \ref{sat12} we can study sequences of graphs of uniformly bounded valence. The connection between graphs and metrics is the observation that to each triangulation of a surface we can associate a Riemannian metric in a more or less canonical way (see Lemma \ref{trian-metric} for details); a related construction has been exploited by Gill-Rohde in \cite{Gill-Rohde} to study weak limits of certain disk triangulations. 

The {\em genus} $g(G)$ of a graph $G$ is the minimal genus of an orientable surface into which $G$ can be embedded. We say that a sequence $(G_i)$ has {\em sublinear genus growth} if 
$$\lim_{i\to\infty}\frac{g(G_i)}{\vert G_i\vert}=0$$
where $\vert G_i\vert$ is the number of vertices of $G_i$. 

\begin{sat}\label{genusgbs}
Let $(G_i)$ be a sequence of graphs with uniformly bounded valence, with $\vert G_i\vert\to\infty$, and with sublinear genus growth. If $(G_i)$ converges in distribution to $\lambda$, then $\lambda$ is supported by the set of 2-parabolic rooted graphs $(G,p)$. In particular, $\lambda$ is supported by rooted graphs with vanishing Cheeger constant $h(G)=0$ and recurrent simple random walk.
\end{sat}

As is the case for manifolds, knowing something about the possible limits of the sequence $(G_i)$ allows us to derive information about the members of the sequence. For example, recall that an {\em expander} is a sequence $(G_i)$ of graphs with uniformly bounded valence such that 
$$\vert G_i\vert\to\infty\ \ \hbox{but}\ \ \liminf_{i\to\infty} h(G_i)>0.$$ 
It follows from \cite[Theorem 4]{sep2} that every expander has linear genus growth. From Theorem \ref{genusgbs} we obtain a new proof of this fact:

\begin{kor}\label{expanders}
For every expander $(G_i)$ there is a positive constant $c>0$ with $g(G_i)\ge c\vert G_i\vert$ for all $i$.
\end{kor}

We describe briefly the strategy of the proof of our main result, Theorem \ref{sat1}. By Theorem \ref{no-collapse} we may assume without loss of generality that all the manifolds $(M_i)$ have injectivity radius at least $10$. Fix a $K$-quasi-conformal map $f_i:M_i\to M$ for all $i$ and choose a maximal 1-net $\CN_i\subset M_i$. We observe that for all $R>0$ and for most $p_i\in\CN_i$ the image $f_i(B(p_i,R,M_i))$ of the ball around $p_i\in M_i$ of radius $R$ has very small diameter. This implies that the scaled manifolds $(\frac1{\diam_M(f_i(B(p_i,1,M_i)))}M,f(p_i))$ converge in the pointed Gromov-Hausdorff topology to $(\BR^d,0)$. Compactness of quasi-conformal maps implies now that the maps
$$f_i:(M_i,p_i)\to \left(\frac1{\diam_M(f_i(B(p_i,1,M_i)))}M,f(p_i)\right)$$
converge to a quasi-conformal map 
$$F:(X,x)\to(\BR^d,0)$$
where $(X,x)$ is the Gromov-Hausdorff limit of the sequence $(M_i,p_i)$. Denoting by $\CN\subset X$ the limit of the nets $\CN_i\subset M_i$ we observe that either $\BR^d\setminus F(X)$ consists of at most a single point or $F(\CN)$ has at least 2 accumulation points. A variant of a lemma due to Benjamini and Schramm \cite{BS} implies that this last situation does not happen if we take the points $p_i\in\CN_i$ at random.
\medskip

This paper is organized as follows: In sections \ref{sec-preli} and \ref{sec:quasi} we discuss mostly well-known facts on Gromov-Hausdorff limits, distributional limits and quasi-conformal maps. In section \ref{sec:no-collapse} we prove Theorem \ref{no-collapse}. In section \ref{sec:bs} we obtain two rather simpleminded manifold versions of the key lemma in \cite{BS}. Once this is done, we prove Theorem \ref{sat1} in section \ref{sec:sat1} and obtain as a consequence Theorem \ref{cheeger} in section \ref{sec:cheeger}. Theorem \ref{sat12} is proved in section \ref{sec:2dim} and Theorem \ref{genusgbs} and Corollary \ref{expanders} in section \ref{sec:graphs}.
\medskip

\noindent{\bf Acknowledgements.} Hossein Namazi and Pekka Pankka thank the Mathematics Department of the University of British Columbia and the Pacific Institute of Mathematics for their hospitality while this paper was being written. Juan Souto thanks Omer Angel, Asaf Nachmias and Gourab Ray for teaching him about the paper \cite{BS}.

\section{Limits}\label{sec-preli}
In this section we fix some notation and remind the reader of a few well-known facts on Gromov-Hausdorff convergence and distributional limits.
\medskip

Abusing notation, we denote by $B(x,r,X)$ both the open and the closed metric balls of radius $r$ centered at $x$ in a metric space $X=(X,d_X)$. A subset $\CN$ of $X$ is {\em $r$-dense} if $X=\bigcup_{x\in\CN}B(x,r,X)$. An {\em $r$-net} in $X$ is a subset $\CN$ with $d_X(x,y)\ge r$ for all pairwise distinct $x,y\in\CN$. Such an $r$-net is {\em maximal} if it is not properly contained in any other $r$-net. Observe that a maximal $r$-net is also an $r$-dense subset of $X$. 

All Riemannian manifolds under consideration in this paper are connected and complete.

\subsection{Gromov-Hausdorff convergence}
Recall that two pointed metric spaces $(X,x)$ and $(Y,y)$ are close to each other in the pointed Gromov-Hausdorff topology if for $R>0$ large and $\epsilon>0$ small there is $L$ close to one and two discrete subsets $U\subset X$ and $V\subset Y$ with $x\in U$, $y\in V$,
$$B(x,R,X)\subset\bigcup_{u\in U}B(u,\epsilon,X),\ \ B(y,R,Y)\subset\bigcup_{v\in V}B(v,\epsilon,Y)$$ and an $L$-bi-Lipschitz map $(U,d_X)\to(V,d_Y)$ mapping $x$ to $y$. We denote by $\CH$ the space of all isometry classes of pointed proper metric spaces endowed with the pointed Hausdorff topology. 

The Bishop-Gromov theorem \cite{Gromov-book} implies that if $M$ is a Riemannian $d$-manifold with bounded sectional curvature $\vert\kappa_M\vert\le 1$ then for all $x\in M$ and all $r\le R$ we have
$$\frac{\vol_M(B(p,R,M))}{\vol_M(B(p,r,M))}\le\frac{\vol_{\BH^d}(B(p,R,\BH^d))}{\vol_{\BH^d}(B(p,r,\BH^d))}.$$
This implies that for every $\epsilon$ and $R$ there is a bound depending only on dimension on the number of points in an $\epsilon$-net which are contained in a given ball of radius $R$. We will use this fact very often later on, but at this point we just want to observe the it implies following well-known fact. Recall that $\CH$ is the space of isometry classes of pointed metric spaces with the Gromov topology.
\medskip

\noindent{\bf Fact.} {\em The subset of $\CH$ consisting of pointed Riemannian manifolds $(M,p)$ of dimension $d$ and pinched curvature $\vert\kappa_M\vert\le 1$ is precompact in $\CH$ and its closure is separable.}
\medskip

The limits in $\CH$ of sequences $(M_i,p_i)$ of pointed Riemannian $d$-dimensional manifolds with pinched curvature satisfy strong regularity properties. In this paper we will be only interested in the limits of sequences for which the injectivity radius $\inj(M_i,p_i)$ at the base point is uniformly bounded from below. In this setting we have the following amazing result due to Gromov:

\begin{named}{$C^{1,1}$-compactness theorem}[Gromov] Fix $d$ and $\epsilon$ and suppose that $(M_i,p_i)$ is a sequence of pointed Riemannian $d$-manifolds satisfying
$$\vert\kappa_{M_i}\vert\le 1\ \ \hbox{and}\ \ \inj(M_i,p_i)\ge\epsilon$$
and converging in $\CH$ to a pointed metric space $(X,x)$. Then $X$ is a smooth manifold endowed with a $C^{1,1}$-Riemannian metric and $(M_i,p_i)$ converges to $(X,x)$ in the $C^{1,\alpha}$-topology for all $\alpha<1$.
\end{named}

Recall that a sequence $(M_i,p_i)$ of pointed Riemannian manifolds converges in the $C^{1,\alpha}$-topology to a pointed Riemannian manifold $(N,p)$ if for all $R>0$ there is a domain $\Omega\subset N$ containing $B_M(p,R)$ and a sequence of maps
$$f_i:(\Omega,p)\to(M_i,p_i)$$
so that the pulled-back metrics converge in the $C^{1,\alpha}$-topology on tensors on $\Omega$ to the restriction to $\Omega$ of the metric of $N$. 

Recall also that a $C^{1,1}$-Riemannian metric on a manifold is one whose first derivatives are Lipschitz continuous. In particular, the second derivatives, and hence the curvature, exist almost everywhere. In any case, the limit $X$ of any sequence $(M_i)$ as in the $C^{1,1}$-compactness theorem has curvature pinched by $\pm 1$ in the sense of Alexandroff. Recall that $X$ is in fact a smooth Riemannian manifold of constant curvature $\kappa$ if the approximating manifolds satisfy $\Vert\kappa_{M_i}-\kappa\Vert_\infty\to 0$ \cite{Gromov-negcur,Petersen-book}. We will need this result in the following form:
\medskip

\noindent{\bf Fact.} {\em Suppose that $(M_i,p_i)$ is a sequence of pointed Riemannian manifolds with $\inj(M_i,p_i)$ uniformly bounded from below and such that $\vert\kappa_{M_i}\vert\le\epsilon_i\to 0$. If $(M_i,p_i)$ converges in $\CH$ to $(X,x)$ then $X$ is a flat Riemannian manifold.}
\medskip

Before moving on we remind the reader that by Cartan's classical theorem, every complete flat manifold is isometric to the quotient of $\BR^d$ by a discrete group of isometries.
\medskip

We refer to \cite{doCarmo} for classical results in Riemannian geometry, to \cite{Gromov-book} for facts and definitions on the Gromov-Hausdorff topology and much more, and to \cite{Petersen-book,Petersen-art} for proofs of the $C^{1,1}$-compactness theorem.

\subsection{Distributional limits}\label{sec:weak}
Suppose now that $(M_i)$ is a sequence of closed Riemannian $d$-manifolds with pinched curvature $\vert\kappa_{M_i}\vert\le 1$. For each $i$ we consider the continuous map
$$\iota_{M_i}:M_i\to\CH,\ \ \iota_{M_i}(p)=(M_i,p).$$
Being a Riemannian manifold, $M_i$ is endowed with a natural measure $\vol_{M_i}$. Pushing forward by $\iota_{M_i}$ the associated probability measure $\frac 1{\vol_{M_i}(M_i)}\vol_{M_i}$ we obtain a sequence of probability measures on $\CH$. As we mentioned earlier, all these measures are simultaneously supported by a compact separable subspace of $\CH$. Hence, the sequence of probability measures $\left(\iota_{M_i}\right)_*\left(\frac 1{\vol_{M_i}(M_i)}\vol_{M_i}\right)$ has a subsequence which converges in the weak-*-topology to some probability measure $\lambda$. If the sequence actually converges, then
$$\lambda=\lim_{i\to\infty}\left(\iota_{M_i}\right)_*\left(\frac 1{\vol_{M_i}(M_i)}\vol_{M_i}\right)$$
is the {\em distributional limit} of the sequence of manifolds $(M_i)$.

Clearly, when trying to prove something about distributional limits of sequences of manifolds one wants to argue in the manifolds themselves. The following lemma whose proof is left to the reader provides this.

\begin{lem}\label{lem-generic}
Suppose that $(M_i)$ is a sequence of Riemannian $d$-manifolds with $\vert\kappa_{M_i}\vert\le 1$ and with distributional limit $\lambda$. Given $b\in[0,1]$ suppose that $(U_i)$ is a sequence of subsets of $M_i$, with $\vol_{M_i}(U_i)\ge b\vol_{M_i}(M_i)$ for all sufficiently large $i$. The set of those $(X,x)\in\CH$ for which there is a sequence $(p_i)$ of points with $p_i\in U_i$ such that a subsequence of $(M_i,p_i)$ converges to $(X,x)$ has at least $\lambda$-measure $b$.\qed
\end{lem}

%\begin{proof}
%As above, consider the embeddings $\iota_{M_i}:M_i\to\CH$ and for each $i$ let $\lambda_i'$ be the push-forward under $\iota_{M_i}$ of the measure $\frac 1{\vol_{M_i}(M_i)}(\chi_{U_i}\vol_{M_i})$ where $\chi_{U_i}$ is the characteristic function of $U_i$. Passing to a subsequence we can assume that the sequence $(\lambda_i')$ converges to a measure $\lambda'$ on $\CH$. Clearly, $\lambda'(V)\le\lambda(V)$ for all $V\subset\CH$ compact. On the other hand, $\lambda'(\CH)\ge b$ because each one of the measures $\lambda_i'$ has at least total mass $b$. The claim follows.
%\end{proof}

\noindent{\bf Example:} {\em Biassed versus unbiassed choice of base points.} We discuss now an example showing that Theorem \ref{sat1} utterly fails if we do not assume that the sequences of base points is chosen by random.

Denote by $\BS^2$ and $\BS^3$ the round spheres of radius 1 in dimension 2 and 3. There are 4 open metric balls $B_1,B_2,B_3,B_4\subset\BS^3$ with disjoint closures and a metric $\rho$ on $\BS^3\setminus\bigcup B_i$ conformal to the restriction of the standard metric of $\BS^3$ and such that each boundary component has a neighborhood isometric to $\BS^2\times[0,1]$; set 
$$Y=(\BS^3\setminus\bigcup B_i,\rho).$$
Let now $B$ be any metric ball in $\BS^3$ and notice again that there is a metric $\rho'$ on $\BS^3 \setminus B$ conformal to the restriction of the standard metric of $\BS^3$ and such that the boundary has a neighborhood isometric to $\BS^2\times[0,1]$; set
$$E=(\BS^3\setminus B,\rho').$$
Consider now the infinite 4-valent tree $T$ and fix a root $t\in T$. For $i\ge 1$ consider $T_i$ the ball in $T$ of radius $i$ centered at $t$. We associate to each $i$ a manifold $M_i$ obtained as follows: For each interior vertex of $T_i$ take a copy of $Y$ and for every terminal vertex a copy of $E$. Then glue all this pieces, via local isometries, according to the combinatorics determined by $T_i$. The manifold $M_i$ is conformally equivalent to $\BS^3$ for all $i$. This can be either seen by constructing directly a conformal diffeomorphism $M_i\to\BS^3$ or can be deduced from the facts that (1) $M_i$ is homeomorphic to $\BS^3$ and hence simply connected and (2) $M_i$ is conformally flat \cite{Kuiper}.

If we choose now base points $p_i\in M_i$ in the $Y$-piece corresponding to the root $t\in T_i$, then any geometric limit $(X,x)$ of the sequence $(M_i,p_i)$ is going to be isometric to a manifold constructed from infinitely many $Y$-pieces glued accordingly to the combinatorics of $T$. The manifold $X$ is conformally equivalent to the complement of a (tame) Cantor set in $\BS^3$.

\begin{figure}[h]
        \centering
         \includegraphics[width=5cm]{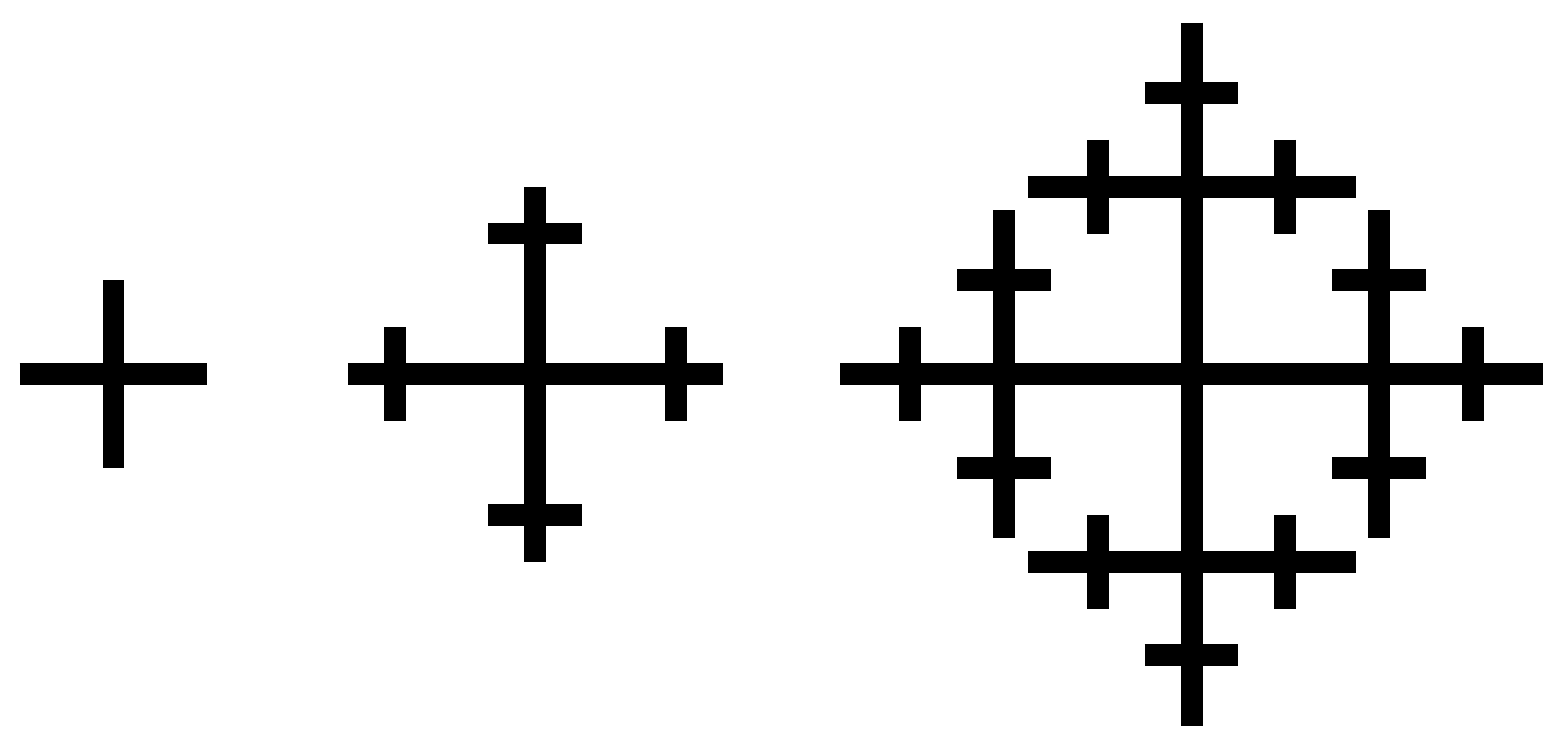}
         \caption{}
\end{figure}

The example we just discussed shows that all the results of this paper just apply for unbiassed choices of base points and not for arbitrary base points. In this example, the reason why the base points $p_i$ are very particular is that in the finite trees $T_i$, the probability of choosing a point within 100 miles of some terminal vertex is overwhelming.
\medskip

\noindent{\bf Example:} {\em Sequences whose distributional limits have infinite topology.}
Let $T'\subset T\subset \BR^3\subset \BS^3$ be a smooth Whitehead pair of solid tori in $\BS^3$, that is, $T'$ and $\BS^3\setminus T$ are tubular neighborhoods of smooth circles $S_1$ and $S_2$ forming a Whitehead link in $\BS^3$, respectively. Let $\psi \colon T \to T'$ be a diffeomorphism which extends smoothly to a neighborhood of $T$, and denote $T_k = \psi^k(T)$ for $k\ge 0$. It is now easy to find an embedding $\phi \colon T \to \BR^4$ which is the identity near $\partial T$ and satisfies $\phi\circ \psi(x) = x+e_4$ on $T$. Similarly, for every $m\ge 0$, there exists an embedding $\phi_m \colon T  \to \BR^4$ which is identity near $\partial T$ and satisfies $\phi_m\circ \psi^k(x) = \phi_m(x)+k e_4$ for all $x\in T_0\setminus T_1$ and all $k\in \{0,\ldots, m-1\}$. 

We fix now Riemannian metrics $g_m$ on $\BS^3$ induced by these partial embeddings, that is, $(T,g_m)$ is isometric to $\phi_m(T)$ and $(\BS^3\setminus T,g_m)$ to $\BS^3\setminus T$. The manifolds $S_m = (\BS^3,g_m)$ have injectivity radius uniformly bounded from below and uniformly pinched sectional curvature.

Since $\vol(\phi_m(T_0\setminus T_m)) = m \vol(\phi(T_0\setminus T_1)) \to \infty$ as $m\to \infty$, we easily conclude that, for every distributional limit $\lambda$ of the sequence $(S_m)$, the limit space $(X,p)$ is $\lambda$-almost surely the infinite Whitehead tower
\[
W = \bigcup_{k\in \BZ} \left( \phi(T\setminus T')+ke_4\right)
\]
having infinite topology.

\begin{bem}
It seems that the notion of distributional limit of a sequence of metric spaces was introduced by Benjamini and Schramm in the context of graphs \cite{BS}. After their seminal result this notion has played an important role in probability theory. However, in the context of Riemannian manifolds it remains, to our knowledge, almost unexplored (see however \cite{samurais}).
\end{bem}

\section{Quasi-conformal maps}\label{sec:quasi}
We recall now some basic facts on quasi-conformal mappings and their compactness properties. All the results in this section are classical for quasi-conformal maps $\BR^d\to\BR^d$ and also exist in some form in the literature for quasi-conformal maps between rather general metric spaces. It is however hard to give concrete references for the statements in the Riemannian setting and this is why we give rather complete proofs. The discussion here parallels Euclidean results in \cite{Vaisala} and \cite{Rickman-book} although we base some of the arguments on the notion of a \emph{Loewner space} introduced by Heinonen and Koskela in \cite{HK}; see also \cite{Heinonen-Book}.
\medskip

Let $M$ and $N$ be Riemannian manifolds. A homeomorphism $f\colon M \to N$ is \emph{quasi-conformal} if there exists $H\ge 1$ so that for all $x\in M$
\begin{equation}
\label{eq:QC_H}
\limsup_{r \to 0} \frac{\max_{d(x,y)=r} d(f(x),f(y))}{\min_{d(x,y)=r} d(f(x),f(y))} \le H <\infty.
\end{equation}
Equivalently, $f$ is in the local Sobolev space $W^{1,n}_{\loc}(M,N)$ of mappings $M\to N$ and there exists $K\ge 1$ so that
\begin{equation}
\label{eq:QC_K}
\Vert df_x\Vert^n \le K \det(df_x)\quad \hbox{for almost every}\ x\in M.
\end{equation}
In particular, we say that $f$ is \emph{$K$-quasi-conformal} if it satisfies \eqref{eq:QC_K} with constant $K\ge 1$. Both of the given definitions extend naturally to local homeomorphisms from $M$ to $N$. 

Before moving on we recall the following consequence of Rauch's comparison theorem \cite{doCarmo}:
\medskip

\noindent{\bf Fact.} For all $\kappa\ge 0$ and $L>1$ there is $R>0$ such that for every Riemannian $d$-manifold $M$ with $\vert\kappa_M\vert\le\kappa$ and every $x\in M$ the restriction 
$$\exp_x:B(0,R,T_xM)\to B(x,R,M)$$
of the exponential map at $x$ to the ball of radius $R$ is locally $L$-bilipschitz and hence an $L^{2d}$-quasi-conformal local homeomorphism.

\subsection{Modulus}
The \emph{conformal modulus of a path family} of a family of paths $\Gamma$ in a $d$-dimensional Riemannian manifold $M$ is
\[
\Mod_d(\Gamma) = \inf_{\rho} \int_M \rho^d \vol_M
\]
where the infimum is taken over all non-negative Borel functions $\rho$ on $M$ satisfying
\[
\int_\gamma \rho \ds \ge 1
\]
for all locally rectifiable paths $\gamma\in \Gamma$; see \cite[Chapter 1]{Vaisala} or \cite{Heinonen-Book}. Based essentially on change of variables, it is now easy to see that a $K$-quasi-conformal homeomorphism $f\colon M \to N$ satisfies for every path family $\Gamma$ in $M$ the inequality
\begin{equation}
\label{eq:QC_Gamma}
\frac{1}{K^{d-1}} \Mod_d(\Gamma) \le \Mod_d(f(\Gamma)) \le K^{d-1} \Mod_d(\Gamma),
\end{equation}
where $f(\Gamma) = \{ f\circ \gamma \colon \gamma \in \Gamma\}$. 

\begin{bem}
It is known that if $f:M\to N$ is a homeomorphism satisfying \eqref{eq:QC_Gamma} then $f$ is actually quasi-conformal, see \cite[Chapter 4]{Vaisala}.
\end{bem}

If $M$ is a Riemannian manifold, $\Omega\subset M$ is open and $E,F\subset\Omega$ are continua, we denote by $\Gamma(E,F;\Omega)$ the family of all paths $\gamma$ connecting $E$ and $F$ in $\Omega$, that is, paths $\gamma:[a,b]\to M$ so that $\gamma(a)\in E$, $\gamma(b)\in F$ and $\gamma(a,b) \subset \Omega$; note that $\Gamma(E,F,\Omega)$ may be empty. When $\Omega=M$ and there is no need to specify the ambient manifold we write $\Gamma(E,F)=\Gamma(E,F,M)$.

Modulus estimates are the basic tool to obtain qualitative information on quasi-conformal embeddings. The two basic estimates in $\BR^d$ are the capacity of an annulus \cite[7.5]{Vaisala}
\begin{equation}
\label{eq:cap}
\Mod_d(\Gamma(\partial B(0,s,\BR^d), \partial B(0,t,\BR^d)))= \vol(\BS^{d-1})\left( \log\frac{t}{s}\right)^{1-d} 
\end{equation}
and the L\"owner type estimate \cite[10.12]{Vaisala}
\begin{equation}
\label{eq:ELoewner}
\Mod_d(\Gamma(E,F;B(0,t,\BR^d)\setminus B(0,s,\BR^d)))\ge C(d) \log \frac{t}{s} 
\end{equation}
for $0<s<t<\infty$, where in \eqref{eq:ELoewner} the sets $E$ and $F$ are continua in $\BR^d$ connecting boundary components of the annulus $B(0,t,\BR^d)\setminus B(0,s,\BR^d)$.

The spaces satisfying an inequality of the type \eqref{eq:ELoewner} are called Loewner spaces. More precisely, we say that $M$ is a $d$-\emph{Loewner space} if there exists a function 
$$\phi_M \colon (0,\infty) \to (0,\infty)\ \ \hbox{with}\ \ \lim_{t\to 0}\phi_M(t)=\infty$$ 
so that 
\begin{equation}
\label{eq:Loewner}
\Mod_d(\Gamma(E,F;M)) \ge \phi_M\left( \Delta(E,F)\right)
\end{equation}
for all continua $E$ and $F$ in $M$; where 
$$\Delta(E,F) = \frac{\min_{x\in E,y\in F}d(x,y)}{\min\{ \diam_M E, \diam_M F\}}$$
is the \emph{relative distance of $E$ and $F$}. A fundamental observation is that $\BR^d$ and all closed Riemannian $d$-manifolds are Loewner spaces that are also {\em Ahlfors $d$-regular}, meaning that there exists a constant $C\ge 1$ so that \[ C^{-1} r^d \le \vol(B(x,r,M)) \le C r^d\] 
for all balls $B(x,r,M)$ with $r$ positive and smaller then the diameter of $M$. Due to the Ahlfors regularity, {\em we may take the function $\phi_M$ to be a decreasing homeomorphism}. 

We refer to the seminal paper of Heinonen and Koskela \cite{HK} for fundamental results on Loewner spaces.

\subsection{Nets and quasi-conformal maps}
Many of the arguments in this paper will rely on discretizing manifolds and using the fact that  quasi-conformal maps behave well with respect to discretizations. The following observation will not surprise any expert:

\begin{lem}\label{lemma:density}
Suppose $X$ is a complete Riemannian $d$-manifold with $|\kappa_X|\le 1$ and $\inj(X)\ge\pi$. Let $\CN\subset X$ be a maximal $1$-net on $X$ and $f\colon X\to \BR^d$ a $K$-quasi-conformal embedding so that $f(X)\ne \BR^n$. Then every point in $\overline{f(X)}\setminus f(X)$ is also an accumulation point of $f(\CN)$.
\end{lem}

Notice that under the assumptions of Lemma \ref{lemma:density} Rauch's comparison theorem implies that there is $K'$ such that for all $x\in X$ the map
$$\exp_x:B(0,3,T_xX)\to B(x,3,X)$$
is $K'$-quasi-conformal and maps $B(0,1,T_xX)$ is $B(x,1,X)$.

\begin{proof}
Given $y\in \overline{f(X)}\setminus f(X)$ let $(x_i)$ be a sequence in $X$ such that $\lim_{i\to\infty}f(x_i)=y$ and choose for each $i$ a point $p_i\in\CN$ so that $d_X(x_i,p_i)\in[0,1]$ and let $p_i'\in B(0,2,T_{x_i}X)$ be such that $\exp_{x_i}(p_i')=p_i$. Consider the $KK'$-quasi-conformal map
$$h_i=f\circ\exp_{x_i}: B(0,3,T_{x_i}X)\to\BR^d$$
A modulus argument (see \cite[18.1]{Vaisala}) shows that there is an increasing function $\theta \colon (0,1) \to (0,\infty)$ depending only on $d$ and $KK'$ so that
$$\frac{|h_i(0)-h_i(p_i')|}{d(h_i(0), \partial h_i(B(0,3,T_{x_i}X)))} \le \theta\left( \frac{|p_i'|}{d(0, \partial B(0,3,T_{x_i}X)))}\right)\le\theta\left( \frac23\right).$$
Because $y\notin f(X)$, we must have
 \[ d(f(x_i), y) \ge d(f(x_i), \D h_i(B(0,3,T_{x_i}X))); \]
 hence
$$| f(x_i)-f(p_i)| = |h_i(0)-h_i(p_i')| \le C(n,K)d(f(x_i),y) \to 0$$
as $i \to 0$, and the points $f(p_i)$ in the image of the net accumulate to $y$ as we needed to prove.
\end{proof}

\subsection{Contraction property}
Throughout this paper we will be considering quasi-conformal maps which on the large scale decrease volume by a large factor. The following proposition asserts that in such a situation the image of most balls of given radius has very small diameter:

\begin{prop}\label{prop:contraction}
For all $d$, $K$ and $\kappa$ there are $C(d,K,\kappa)$ and $\epsilon_0(d,K,\kappa)$ such that the following holds: If $f \colon M' \to M$ is a $K$-quasi-conformal homeomorphism between two Riemannian $d$-manifolds satisfying 
$$|\kappa_{M'}|\le 1,\ \ \inj(M')\ge 10\ \ \hbox{and}\ \ |\kappa_M|\le\kappa,$$ 
then for every $1$-net $\CN\subset M'$,
$$\vert\{ p \in\CN \vert \diam_M(f(B(p,R,M'))) > \varepsilon\}\vert \le C(d,K,\kappa) \vol_M(M) R^d\varepsilon^{-d}e^{(d-1)R} $$
for every $R\ge 1$ and every $\varepsilon\in(0,\epsilon_0(d,K,\kappa))$.
\end{prop}

For the remaining of this section we fix $d$, $K$ and $\kappa$.
\medskip

The proof of Proposition \ref{prop:contraction} has its roots in the classical local quasisymmetry property of quasi-conformal maps. Let $f\colon B(0,r,\BR^d) \to B(0,R,\BR^d)$ be a $K$-quasi-conformal embedding. Then, suppressing from our notation the reference to $\BR^d$, there exists a constant $C_1=C_1(d,K)\ge 1$ so that
\begin{equation}
\label{eq:round}
B\left(f(0),\frac{\diam_{\BR^d}(f(B(0,r)))}{C_1}\right)\subset f\left(B\left(0,\frac r2\right)\right).
\end{equation}
Indeed, with $\delta_1=\diam_{\BR^d}(f(B(0,r)))$ and $\delta>0$ maximal with $B^d(f(0),\delta)\subset f(B(0,\frac r2))$ we have
\begin{eqnarray*}
\Mod_d(\Gamma(\partial B(f(0),\delta), \partial B(f(0),\delta_1))) &\le \Mod_d(\Gamma(\partial f(B(0,\frac r2)), \partial f(B^d(0,r)))) \\
&\le K^{d-1} \Mod_d(\Gamma(\D B(0,r/2),\partial B(0,r))). 
\end{eqnarray*}
The estimate \eqref{eq:round} follows now from \eqref{eq:cap}.

\begin{lem}\label{lemma:contraction}
There exist constants $\varepsilon_0=\varepsilon_0(d,K,\kappa)>0$ and $C_2=C_2(d,K)$ such that if $M$ is a Riemannian $d$-manifold with $|\kappa_M|<\kappa$, if $f\colon B(0,r,\BR^d) \to M$ is a $K$-quasi-conformal embedding and if  $\vol_M(f(B(0,r,\BR^d)))<\varepsilon_0$ then
$$\diam_M(f(B(0,r,\BR^d))) \le C (\vol f(B(0,r/2,\BR^d)))^{1/d}.$$
\end{lem}

\begin{proof}
Choose $R=R(\kappa)>0$ so that the exponential map $\exp_x \colon B(0,R,T_xM) \to B(x,R,M)$ is locally $2$-bilipschitz for all $x\in M$.

For some $\varepsilon_0>0$ to be determined consider a $K$-quasi-conformal map 
$$f\colon B(0,r,\BR^d) \to M$$ 
with $\vol_M(f(B(0,r,\BR^d)))<\varepsilon_0$ and set $p=f(0)$. Choose also $\delta\in(0,r]$ maximal satisfying 
$$\diam_M(f(B(0,\delta,\BR^d))) \le R$$ 
and notice that there is $\tilde f\colon B(0,\delta,\BR^d) \to B(0,R,T_pM)$ so that the following diagram commutes:
$$\xymatrix{ & (B(0,R,T_pM),0)\ar[d]^{\exp_p} \\
(B(0,\delta,\BR^d),0)\ar[r]_f\ar[ru]^{\tilde f} & (B(p,R,M),p)}$$
Note that the exponential map $\exp_p$ is an embedding from $\tilde f(B(0,\delta,\BR^d))$ to $f(B(0,\delta,\BR^d))$ and therefore is 2-bilipschitz.
We hence have that
\begin{align*}
\varepsilon_0>\vol_M(f(B(0,\delta/2,\BR^d))) & \ge 2^{-d} \vol_M(\tilde f(B(0,\delta/2,\BR^d)))\\
& \ge 2^{-d}C \diam_M(\tilde f(B(0,\delta,\BR^d)))^d \\
&\ge 2^{-d-1}C \diam_M(f(B(0,\delta,\BR^d)))^d
\end{align*}
where the first (interesting) and third inequalities hold because $\exp_x$ is 2-bilischitz on $\tilde f(B(0,\delta,\BR^d))$; the second inequality follows from \eqref{eq:round} and the fact that $\tilde f$ is $4^dK$-quasi-conformal and $C$ is a constant depending on $d$ and $K$.

The claim follows now, from the observation that $\delta=r$ as long as $\varepsilon_0\le 2^{-d-1}C R^d$.
\end{proof}

\begin{lem}\label{prop:diam}
There exist positive constants $\varepsilon_0(d,K,\kappa)$ and $C_3=C_3(d,K,\kappa)$ so that the following holds: If $M$ and $M'$ are Riemannian $d$-manifolds with $|\kappa_{M'}|\le 1$, $\inj M' \ge 10$ and $|\kappa_M|\le \kappa$, if $\CN$ is a $\delta$-net for $\delta < 3$ in $M'$, and if $f:M'\to M$ is a $K$-quasi-conformal homeomorphism then 
$$\vert\{ x\in\CN \vert \diam_M(f(B(x,2 \delta,M'))) > \varepsilon\}\vert \le C_3(d,K,\kappa)\frac{\vol_M(M)}{\varepsilon^d}$$
for all $\varepsilon< \varepsilon_0(d,K,\kappa)$.
\end{lem}

\begin{proof}
Let $\varepsilon_0(d,K,\kappa)$ be as in Lemma \ref{lemma:contraction}, choose $\varepsilon<\varepsilon_0(d,K,\kappa)$ and set
$$\CN_\varepsilon = \{ x \in\CN \vert \diam_M(f(B(x,2\delta,M'))) > \varepsilon\}.$$
By Lemma \ref{lemma:contraction} and the fact that the exponential map $\exp_x:B(0,2\delta,T_x M')\to B(x,2\delta,M')$ is $L$-bilipschitz for some $L$ depending only on the constants in the statement of Lemma \ref{prop:diam}, there is a constant $C$ depending on $d, K$ and $\kappa$ so that for every $x\in\CN_\varepsilon$
$$C \vol_M(f(B(x,\delta,M'))) \ge \diam_M(f(B(x,2\delta,M')))^d > \varepsilon^d,$$
or 
\[ \vol_M(f(B(x,2\delta,M')) \ge \varepsilon_0 > \varepsilon_0^d \ge \varepsilon^d; \]
in any case for every $x\in \CN_\varepsilon$,
\[ C' \vol_M(f(B(x,2\delta,M')) > \varepsilon^d \]
with $C'$ depending on $d,K$ and $\kappa$.
Since the balls of radius $\frac\delta 2$ centered at points of $\CN$ are pairwise disjoint, it follows from the Bishop-Gromov theorem that every point of $M$ is contained in at most $C''$ balls $B(x,2\delta,M')$ with $C''$ depending only on $\kappa$ and $d$. Therefore
\begin{align*}
\varepsilon^d\vert\CN_\varepsilon\vert &\le \sum_{x\in\CN_\varepsilon} C' \vol_M(f(B(x,\delta/2,M'))) \\
&\le C'' C' \vol_M(M).
\end{align*}
This proves the claim.
\end{proof}

We are now ready to prove Proposition \ref{prop:contraction}:

\begin{proof}[Proof of Proposition \ref{prop:contraction}]
Let $\varepsilon_0(d,K,\kappa)$ be as provided by Lemma \ref{prop:diam} and for $\varepsilon<\varepsilon_0(d,K,\kappa)$ consider the following subsets
$$\CN' = \{ p \in\CN \vert \diam_M(f(B(p,R,M'))) > \varepsilon\}\ \ \hbox{and}$$
$$\CN'' = \{ p \in\CN \vert \diam_M(f(B(p,2,M'))) > \varepsilon/(2R)\}$$
of the net $\CN$. Notice that Lemma \ref{prop:diam} now yields $\vert\CN''\vert\le C_3(d,K,\kappa) 2^d R^d \frac{\vol_M(M)}{\varepsilon^d}$.

Suppose that we have $p\in\CN'$, fix $x\in \partial B(p,R,M)$ so that 
$$d_M(f(p),f(x)) \ge \frac{\diam_M(f(B(p,R,M')))}2$$
and let $\gamma$ be the geodesic from $p$ to $x$. Since $\CN$ is an $1$-net we find points $p_1,\dots,p_k\in\CN\cap B(p,R,M')$ with $k\le R$ with $d_{M'}(p_i,\gamma)\le 1$ such that the balls $B(p_i,2,M')$ cover $\gamma$. Since $\diam_M(f(\gamma)) > \frac{\varepsilon}2$, it follows that for some $i$ we have $\diam_M(f(B(p_i,2,M')))> \varepsilon/(2R)$. This proves that for each $p\in\CN'$ there is $q\in\CN''\cap B(p,R,M')$. Since, by the Bishop-Gromov theorem, any ball in $M'$ of radius $1$ is contained in at most $C(d)e^{(d-1)R}$ balls of radius $R$ centered at points of $\CN'$ we deduce that 
$$\vert\CN'\vert\le C(d)C_3(d,K,\kappa)e^{(d-1)R} 2^d R^d \frac{\vol_M(M)}{\varepsilon^d}$$
which is, after renaming constants, what we needed to prove.
\end{proof}

\subsection{Compactness properties of quasi-conformal maps}
We discuss now the compactness properties of quasi-conformal maps needed in this paper. See \cite[Section 21]{Vaisala} for the corresponding results for maps between domains of euclidean space.

\begin{lem}
\label{lemma:equicontinuity}
Suppose that $X$ is a complete Riemannian $d$-manifold with $\vert\kappa_X\vert\le 1$ and $\inj(X)>0$, $\Omega\subset X$ is a domain, and let $M$ be a closed Riemannian $d$-manifold.

Suppose $\CF$ is a family of $K$-quasi-conformal embeddings $\Omega\to M$ so that for some uniform $\delta>0$ and for every $f\in \CF$ there are points $p_f,q_f\in M\setminus f(X)$ with $d(p_f,q_f)\ge \delta>0$. Then $\CF$ is equicontinuous. 
\end{lem}

Before proving Lemma \ref{lemma:equicontinuity} we remind the reader that the Loewner function 
$\phi_M \colon (0,\infty) \to (0,\infty)$ is a monotonously decreasing homeomorphism.

\begin{proof}
We need to show that for every $x_0\in\Omega$ there are $s_0>0$ and a function $\psi \colon (0,s_0] \to (0,\infty)$ with $\lim_{s \to 0} \psi(s) = 0$ such that $\diam_M(f(B(x_0,s,X))) < \psi(s)$ for all $f\in\CF$ and $s\le s_0$. To begin with let $t>0$ be so that the exponential mapping 
$$\exp_{x_0}:B(0,t,T_{x_0}X)\to B(0,t,X)\subset\Omega$$ 
is a $2$-quasi-conformal embedding. For $s<t$, let $\Gamma_{st}$ be the family of paths connecting $\partial B(x_0,s,X)$ to $\D B(x_0,t,X)$ in $X$ and notice that \eqref{eq:QC_Gamma} and \eqref{eq:cap} imply that there is $C(d)>0$ so that
$$\Mod_d(\Gamma_{st}) \le C(d) \left( \log \frac{t}{s} \right)^{1-d}.$$
Let $\phi_M$ be a Loewner function for $M$ and notice that we can fix $s_0<t$ with 
$$C(d)K \left( \log \frac{t}{s} \right)^{1-d} \le \phi_M\left(\frac{2\diam M}{\delta}\right)$$
for all $s<s_0$. 

Given $f\in\CF$ consider the two continua 
$$E = f(B(x_0,s,X))\ \ \hbox{and}\ \ F= M\setminus f(B(x_0,t,X)),$$
and notice that $F$ is connected with $\diam F \ge d(p,q)\ge \delta$. From \eqref{eq:Loewner} we obtain that 
\begin{eqnarray*}
\Mod_d(\Gamma(E,F;M)) &\ge& \phi_M\left(\frac{d(E,F)}{\min\{\diam E, \diam F\}}\right) \\
&\ge& \phi_M\left(\frac{\diam M}{\min\{\diam E, \delta\}}\right).
\end{eqnarray*}
because $\phi_M(\cdot)$ is decreasing. 

Since $f$ is a $K$-quasi-conformal embedding and every path in $\Gamma(E,F;M)$ has a subpath in $f(\Gamma_{st})$, we have
\begin{eqnarray*}
\phi_M\left(\frac{\diam M}{\min\{\diam E, \delta\}}\right) &\le& \Mod(f(\Gamma_{st})) \le C(d)K \left( \log\frac{t}{s} \right)^{1-d} \\
&\le& \phi_M\left(\frac{2\diam M}{\delta}\right).
\end{eqnarray*}
Again using that $\phi_M(\cdot)$ is decreasing we deduce that $\diam_M(E)<\delta/2$. Thus
$$\frac{\diam_M(M)}{\diam_M(E)}\ge\phi_M^{-1}\left( C(d)K \left( \log \frac{t}{s} \right)^{1-d}\right)$$
for all $s\le s_0$. This implies that 
$$\diam_M(f(B(x_0,s,X))) = \diam E \le\frac{\diam M}{\phi_M^{-1}\left( C(d)K\left(  \log \frac{t}{s} \right)^{1-d}\right)} =: \psi(s).$$
Since $\psi(s)$ tends to $0$ when $s\to 0$, this proves the equicontinuity of the family $\CF$.
\end{proof}

We state now some useful consequences of Lemma \ref{lemma:equicontinuity}.

\begin{kor}
\label{cor:equicont}
Let $X$ and $M$ be as in Lemma \ref{lemma:equicontinuity} and suppose $\CF$ is a family of $K$-quasi-conformal embeddings $X\to M$ so that there are $x_1,x_2\in X$ and $\delta>0$ such that for every $f\in \CF$ there exists a point $p_f\in M\setminus f(X)$ with $d_M(p_f,f(x_i))\ge \delta>0$ for $i=1,2$. Then $\CF$ is equicontinuous. 
\end{kor}

\begin{proof}
Let $X_i = X\setminus \{x_i\}$ for $i=1,2$. Then $f|_{X_i}$ omits points $p_f$ and $f(x_i)$ for every $f\in \CF$. Thus the family $\{ f|_{X_i} \colon f\in \CF\}$ is equicontinuous for $i=1,2$. Hence $\CF$ is equicontinuous on $X=X_1\cup X_2$.
\end{proof}

\begin{kor}
\label{cor:qc_limit}
Let $X$ be a complete Riemannian manifold and $\Omega_1 \subset \overline{\Omega_1} \subset \Omega_2\subset \cdots$ an exhaustion of $X$ by bounded domains. Then a sequence $f_i \colon \Omega_i \to \BR^d$ of $K$-quasi-conformal embeddings has a subsequence converging to a $K$-quasi-conformal embedding $f \colon X\to \BR^d$ if there exists $x,y\in X$ so that sequences $(f_i(x))$ and $(f_i(y))$ converge in $\BR^d$ to different points. 
\end{kor}
\begin{proof}
Indeed, since $\BR^d$ embeds conformally into $\BS^d$, we may consider $(f_i)$ as a sequence $f_i \colon \Omega_i \to \BS^d$. Since all maps $f_i$ omit the point at infinity and sequences $(f_i(x))$ and $(f_i(y))$ converge, we have that $(f_i)$ is equicontinuous by Corollary \ref{cor:equicont} on every $\Omega_i$ and hence normal by the Arzela-Ascoli theorem. Thus, for every $\Omega_i$ we have a subsequence of $(f_i)$ converging to a continuous map $f_{i,\infty}:\Omega_i\to\BS^d$. By a diagonal argument, we may assume that $f_{i,\infty}$ extends $f_{i-1,\infty}$ for all $i>0$ and hence we get $f:X\to\BR^d$. Notice now that the map $f$ is not constant because $f(x)\neq f(y)$. It follows hence from \cite[21.1]{Vaisala} that in fact $f$ is a $K$-quasi-conformal embedding.
\end{proof}

Combining the arguments in the proofs of Corollary \ref{cor:equicont} and Corollary \ref{cor:qc_limit} we also obtain:

\begin{kor}\label{cor:qc_limit2}
Let $X$ and $M$ be complete Riemannian manifolds and $\Omega_1 \subset \overline{\Omega_1} \subset \Omega_2\subset \cdots$ an exhaustion of $X$ by bounded domains. Then a sequence $f_i \colon \Omega_i \to M$ of $K$-quasi-conformal embeddings has a subsequence converging to a $K$-quasi-conformal embedding $f \colon X\to M$ if there exists $x,y,z\in X$ so that sequences $(f_i(x))$, $(f_i(y))$ and $(f_i(z))$ converge in $M$ to different points. \qed
\end{kor}

\section{Proof of Theorem 1.3}\label{sec:no-collapse}

%\marginpar{\tiny H: Something needed to be said about local homeomorphism here. Should we make any references? Quasi-regular maps, etc?}

Note that we had defined quasi-conformal maps to be homeomorphisms. In this section we consider maps that are only locally homemorphisms. Such maps in dimensions $d \ge 3$ have surprising rigidity properties. In Euclidean spaces the seminal theorem is Zorich's global homeomorphism theorem \cite{Z1} (or \cite{Rickman-book}):

\begin{sat*}[Zorich]
For $d\ge 3$ a quasi-conformal local homeomorphism $\BR^d \to \BR^d$ is a homeomorphism.
\end{sat*}

A geometric version of Zorich's theorem is due to Gromov and Zorich \cite{Gromov-book, Z2}:

\begin{sat*}[Gromov-Zorich]
If $d\ge 3$ and $N$ is a simply connected Riemannian $d$-manifold, then every quasi-conformal local homeomorphism $\BR^d \to N$ is an embedding.
\end{sat*}

\begin{bem}
The Gromov-Zorich theorem is in fact more general: one can replace $\BR^d$ by any $d$-parabolic manifold.
\end{bem}

The first step of the proof of Theorem \ref{no-collapse} is the following consequence of the Gromov-Zorich theorem:

\begin{prop}\label{prop:Z}
Let $d\ge3$ and $M$ be a Riemannian $d$-manifold that admits a quasi-conformal embedding $f:X\to M$ from an open flat Riemannian manifold. Then either $f$ is a quasi-conformal homeomorphism or $M$ is quasi-conformal to $\BS^d$ and $X=\BR^d$.
\end{prop}

Before launching the proof of Proposition \ref{prop:Z} we establish a fact which is again surely well-known to experts:

\begin{lem}
\label{lemma:image}
Let $M$ be a Riemannian manifold and $f \colon \BR^d \to M$ a quasi-conformal embedding. Then $f$ is a homeomorphism if $M$ is open. If $M$ is closed, then $M$ is quasi-conformal to $\BS^d$ and $\vert M\setminus f(\BR^d)\vert = 1$.
\end{lem}
\begin{proof}
Suppose first that $M$ is closed. We show that $M$ is a one point compactification of $f(\BR^d)$. Once this has been proved, $f$ extends to a $K$-quasi-conformal homeomorphism $\BS^d\to M$ \cite[17.3]{Vaisala}. Seeking a contradiction suppose that there are two distinct points $y_1,y_2 \in \overline{f(\BR^d)}\setminus f(\BR^d)$ and set $r=d_M(y_1,y_2)$. Then there are for $j=1,2$ proper paths $\gamma_j\colon [0,\infty) \to \BR^d$ with $(f\circ\gamma_j)([0,\infty)) \subset B(y_j,r/4,M)$ and with $\lim_{t\to\infty}(f\circ\gamma_j)(t)=y_j$. Then 
\begin{multline}
\Mod_d(\Gamma( f(\gamma_1[0,t]),f(\gamma_2[0,t]));M) \le \\ 
\le \Mod_d( \Gamma( \partial B(y_1,r/4,M), \partial B(y_2,r/4,M);M) < \infty
\end{multline}
for all $t$. On the other hand, 
$$\Mod_n(\Gamma( \gamma_1[0,t],\gamma_2[0,t])) \to \infty$$
because $\gamma_j(t)$ tends to $\infty$ as $t$ grows (compare with \eqref{eq:ELoewner}). This contradicts the quasi-conformality of $f$. We have proved Lemma \ref{lemma:image} for closed targets.

The proof for open manifolds is similar. Suppose $f(\BR^d)\ne M$ and let $y\in\overline{f(\BR^d)}\setminus f(\BR^d)$ and fix $r>0$ small. We fix now a proper path $\gamma_1 \colon [0,\infty) \to \BR^d$ so that $(f\circ\gamma_1)[0,\infty) \subset B(y,r/2,M)$ and with $\lim_{t\to\infty}(f\circ\gamma_1)(t)=y$. We fix also a second path $\gamma_2 \colon [0,\infty) \to \BR^d$ so that $f\circ \gamma_2$ is proper and has image disjoint of $B(y,r,M)$; this is possible because $M$ has at least one end. The claim now follows using the same modulus argument as in the closed case.
\end{proof}

\begin{proof}[Proof of Proposition \ref{prop:Z}]
Let $X$ be a flat Riemannian $n$-manifold that admits a quasi-conformal embedding $f:X\to M$. Since $X$ is flat, there exists a conformal covering map $\pi:\BR^d \to X$. Thus we have the quasi-conformal local homeomorphism $f\circ\pi\colon \BR^d \to M$. Let $\tilde M$ be the universal cover of $M$ and fix a lift $\widetilde{f\circ\pi}:\BR^d\to\tilde M$ of $f\circ\pi$. By the Gromov-Zorich theorem, $\widetilde{f\circ\pi}$ is an embedding. Thus, by Lemma \ref{lemma:image}, either $\widetilde{f\circ\pi}$ is a quasi-conformal homeomorphism or $\tilde M$ is quasi-conformal to $\BS^d$ and $\widetilde{f\circ\pi}$ omits one point. In the former case, it follows that also $f$ is a quasi-conformal homeomorphism. Suppose that $\tilde M$ is quasi-conformal to $\BS^d$. Compactness of $\tilde M$ implies that the cover $\tilde M\to M$ is finite. However, every flat manifold but $\BR^d$ has infinite fundamental group. Since $f$ was an embedding to begin with we obtain that actually $X=\BR^d$. Now, Lemma \ref{lemma:image} shows that $M$ is quasi-conformal to $\BS^d$, as we needed to prove.
\end{proof}

After this preparatory work we are ready to prove Theorem \ref{no-collapse}.

\begin{named}{Theorem \ref{no-collapse}}
Let $M$ be a closed Riemannian manifold of dimension $d\ge 3$. For every $K\ge 1$ there is $\epsilon$ positive such that every $M'\in\CQ(M,K)$ with diameter $\diam(M')\ge 1$ has injectivity radius $\inj(M')\ge\epsilon$.
\end{named}

Recall that $M'$ belongs to $\CQ(M,K)$ if it has pinched curvature $\vert\kappa_{M'}\vert\le 1$ and if there is a $K$-quasi-conformal homeomorphism $f:M'\to M$.

\begin{proof}
Seeking a contradiction, suppose there exists a sequence $(M_i)\subset\CQ(M,K)$ with $\diam(M_i)\ge 1$ and $r_i = \inj(M_i)\to 0$ as $i\to \infty$. Fix for every $i$ a point $p_i \in M_i$ with $\inj (M_i,p_i) = \inj(M_i)$ and a $K$-quasi-conformal homeomorphism $f_i \colon M_i \to M$.

We consider now pointed manifolds $(N_i,p_i) = (\frac{10}{r_i} M_i, p_i)$ for all $i$. We have $\inj(N_i)=10$ for all $i$ while $|\kappa_{N_i}| \to 0$ and $\diam N_i \to \infty$ as $i \to \infty$. By the Gromov-Hausdorff compactness, we may hence assume, after possibly passing to a subsequence, that $(N_i,p_i)$ converges in the $C^{1,\alpha}$-topology to a $(X,p)$, where $X$ is a non-compact flat manifold. Furthermore, as in the proof of Cheeger's Lemma (see e.g.\;\cite{Petersen-art}), $X$ is not simply connected. 
\medskip

\noindent{\bf Claim.} There is a quasi-conformal embedding of $X$ either into $M$ or into $\BR^d$.
\medskip

Assuming the claim we conclude the proof: in both cases, Proposition \ref{prop:Z} implies that $X$ is homeomorphic to $\BR^d$ and hence that $\pi_1(X)=1$; a contradiction. It remains to prove the claim.
\medskip

Consider a nested exhaustion $\Omega_1\subset\bar\Omega_1\subset\Omega_2\subset\dots$ of $X=\bigcup_i\Omega_i$ by bounded domains such that $\Omega_1$ contains the ball $B(p,1,X)$. By passing to a subsequence, we may assume that for all $i$ there is a $2$-bilipschitz embedding $\psi_i \colon (\Omega_i,p) \to (N_i,p_i)$. Consider the $K2^d$-quasi-conformal maps $h_i = f_i \circ \psi_i \colon \Omega_i \to M$. 

Suppose for the time being that the family $\{h_i:\Omega_i\to M\}$ is equicontinuous. Since $M$ is closed we may assume by the Arzela-Ascoli theorem, and possibly passing to a subsequence, that the maps $h_i$ converge to a continuous map $h:X \to M$. If $h$ is a quasi-conformal embedding we are done, so we may assume by Corollary \ref{cor:qc_limit2} that $h$ is constant; say $h(\Omega_k)=q$. Notice that this implies that $\lim_{i\to\infty}\diam_M(h_i(\Omega_k))=0$ for all $k$. Passing to a subsequence we can hence assume that in fact $\lim_{i\to\infty}\diam_M(h_i(\Omega_k))=0$. Consider now $R>0$ small enough so that 
$$\exp_q:B(0,R,T_qM)\to B(0,R,M)$$
is a 2-quasi-conformal embedding and consider the maps 
$$\exp_q^{-1}\circ h_i:\Omega_k\to T_qM=\BR^d.$$
Corollary \ref{cor:qc_limit} implies that we can post-compose each one of these maps with a conformal automorphism $\phi_i$ of $\BR^d$ so that the maps
$$\phi_i\circ\exp_q^{-1}\circ h_i:\Omega_k\to\BR^d$$
after taking  a diagonal subsequence converge to an embedding $X\to\BR^d$. Thus we are done for $h$ constant. 

It remains to prove the claim in the case that the family $\{h_i:\Omega_i\to M\}$ is not equicontinuous. Under this assumption Lemma \ref{lemma:equicontinuity} and Corollary \ref{cor:qc_limit2} imply that $\lim_{i\to\infty}\diam_M(M\setminus h_i(B(p,1,X)))=0$. Passing to another subsequence we may hence assume that the sets $M\setminus h_i(B(p,1,X))$ converge in the Hausdorff topology to some $q\in M$. In particular we may assume, up to composing the map $h_i:\Omega_i\to M$ with an isotopy which is arbitrarily $C^\infty$-close to $\Id_M$, and hence up to replacing $h_i$ by a say $2K2^d$-quasi-conformal map, that 
$$q\notin\overline{h_i(\Omega_i)}$$ 
for all $i$. Then, for all $i$ there is $k$ with $h_i(\Omega_i)\subset h_k(B(p,1,X))$. Passing yet again to a subsequence we may thus assume that $h_i(\Omega_i)\subset h_{i+1}(B(p,1,X))$ for all $i$. 

Now observe that since $X$ is flat and since $\inj(X)=10$ we have that the ball $B(p,1,X)$ is isometric to the standard ball $B(0,1,\BR^d)$. We have hence the sequence of maps
$$h_{i+1}^{-1}\circ h_i:\Omega_i\to \BR^d.$$
Corollary \ref{cor:qc_limit} implies again that we can post-compose each one of these maps with a conformal automorphism $\phi_i$ of $\BR^d$ so that the maps
$$\phi_i\circ h_{i+1}^{-1}\circ h_i:\Omega_i\to\BR^d$$
converge, after passing for a last time to a subsequence, to an embedding $X\to\BR^d$. This concludes the proof of the claim and hence of Theorem \ref{no-collapse}.
\end{proof}

\section{The Benjamini-Schramm lemma}\label{sec:bs}
In this section we discuss briefly a lemma due to Benjamini and Schramm \cite[Lemma 2.3]{BS} and derive two rather simpleminded manifold versions.

Let $M$ be a Riemannian manifold and $C\subset M$ a finite set of points. The {\em isolation radius} $\rho_{M,C}(w)$ of $w\in C$ is the minimal distance to a different point in $C$:
$$\rho_{M,C}(w)=\min_{z\in C,\ z\neq w}d_M(w,z).$$
Given $\delta\in(0,1)$ and $s>0$, we say that $w\in C$ is {\em $(\delta,s)$-supported} if 
$$\min_{p\in M}\left\vert C\cap\left(B(w,\delta^{-1}\rho_{M,C}(w),M)\setminus B(p,\delta\rho_{M,C}(w),M)\right)\right\vert\ge s.$$
The paradigmatic example of a $(\delta,s)$-supported point $w$ is one for which the ball $B(w,\delta^{-1}\rho_{M,C}(w),M)$ contains two disjoint balls $B_1$ and $B_2$ of radius $\delta\rho_{M,C}(w)$, each one of which contains at least $s$ points in $C$.

\begin{figure}[h]
        \centering
         \includegraphics[width=4cm]{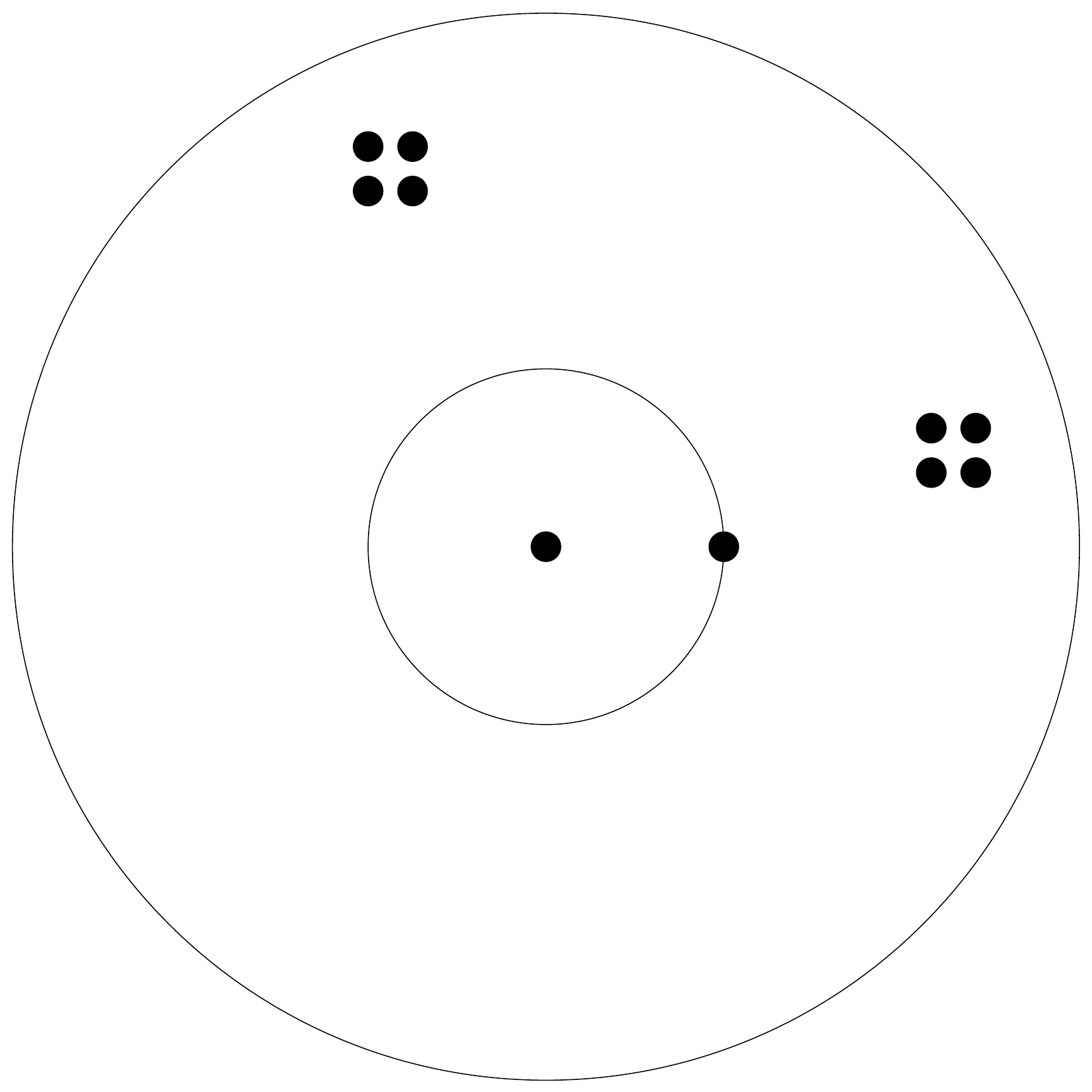}
         \caption{A $(\frac 13,4)$-supported point.}\label{fig1}
\end{figure}

Notice that if $\delta'<\delta$ and if a point is $(\delta,s)$-supported, then it is also $(\delta',s)$-supported. Similarly, if $s>s'$ and if a point is $(\delta,s)$-supported, then it is also $(\delta,s')$-supported.

\begin{lem*}[Benjamini-Schramm]
For every $d>0$ and every $\delta\in(0,1)$ there is a constant $c(d,\delta)$ such that for every finite subset $C$ of $\BR^d$ and every $s\ge 2$ the set of $(\delta,s)$-supported points in $C$ has cardinality at most $c(d,\delta)\frac{\vert C\vert}s$.
\end{lem*}

In \cite{BS}, Benjamini and Schramm proved this lemma only for $d=2$, but the proof applies word-by-word to sets in $\BR^d$ for arbitrary $d$ \cite{BC}. In this section we derive two rather simpleminded versions of the Benjamini-Schramm lemma for manifolds. We first use the Whitney embedding theorem to reduce to the euclidean situation:

\begin{lem}\label{bs1}
Let $M$ be a compact Riemannian manifold. For every $\delta\in(0,1)$ there is a constant $c(M,\delta)$ such that for every finite subset $C$ of $M$ and every $s\ge 2$ the set of $(\delta,s)$-supported points in $C$ has cardinality at most $c(M,\delta)\frac{\vert C\vert}s$.
\end{lem}
\begin{proof}
Consider a smooth embedding $M\hookrightarrow\BR^d$ for some suitable $d$ and let $L\ge 1$ be such that for all $x,y\in M$ we have
$$L^{-1}d_M(x,y)\le d_{\BR^d}(x,y)\le Ld_M(x,y)$$
Suppose now that $C\subset M$ is a finite set and notice that for all $w\in C$ we have the following relation between the separation radius of $w$ when we consider $C$ as a subset of $M$ and as a subset of $\BR^d$:
$$L^{-1}\rho_{M,C}(w)\le\rho_{\BR^d,C}(w)\le L\rho_{M,C}(w).$$
In particular, we have for all $\delta\in(0,1)$ that
\begin{equation}\label{hola1}
B(w,(2\delta L^{2})^{-1}\rho_{M,C}(w),M)\subset B(w,\delta^{-1}\rho_{\BR^d,C}(w),\BR^d).
\end{equation}
Similarly, for all $p\in\BR^d$ and $x,y\in C\cap B(p,\delta\rho_{\BR^d,C}(w),\BR^d)$ we also have
$$d_M(x,y)\le Ld_{\BR^d}(x,y)\le 2L\delta\rho_{\BR^d,C}(w)\le 2L^2\delta\rho_{M,C}(w)$$
and hence that
\begin{equation}\label{hola2}
C\cap B(p,\delta\rho_{\BR^d,C}(w),\BR^d)\subset C\cap B(x,2\delta L^2\rho_{M,C}(w),M).
\end{equation}
Taken together, \eqref{hola1} and \eqref{hola2} imply that for all $p\in\BR^d$ and all $x\in C\cap B(p,\delta\rho_{\BR^d,C}(w),\BR^d)$ we have
\begin{multline}
\left\vert C\cap\left(B(w,\delta^{-1}\rho_{\BR^d,C}(w),\BR^d)\setminus B(p,\delta\rho_{\BR^d,C}(w),\BR^d)\right)\right\vert \\ 
\ge\left\vert C\cap\left(B(w,(2\delta L^2)^{-1}\rho_{M,C}(w),M)\setminus B(x,2\delta L^2\rho_{M,C}(w),M)\right)\right\vert.
\end{multline}
In particular, it follows that if $w\in C$ is $(\delta,s)$-supported in $M$, then is also $(\frac\delta{2L^2},s)$-supported in $\BR^d$. The claim follows now from the Benjamini-Schramm lemma.
\end{proof}

Lemma \ref{bs1} is the form of the Benjamini-Schramm lemma that we will use to prove Theorem \ref{sat1}. Unfortunately, we do not see how to make use of Lemma \ref{bs1} to prove Theorem \ref{sat12} because we need to apply a version of the Benjamini-Schramm lemma to a sequence of pairwise distinct surfaces. What will come to our help is that for surfaces one can explicitly construct nice atlases.

\begin{lem}\label{bs2}
Let $M$ be a compact $d$-manifold, $U_1,\dots,U_r$ a collection of open sets covering $M=\bigcup U_i$, $\phi_i:U_i\to\BR^d$ an embedding for each $i$ and let $k$ be the multiplicity of the covering. 

For every $\delta\in(0,1)$ and every finite subset $C$ of $M$ and every $s\ge 2$ the set of those $x\in C$ for which there is $i$ with $x\in U_i$ and such that $\phi_i(x)$ is $(\delta,s)$-supported in $\phi_i(C\cap U_i)$ has cardinality at most $c(d,\delta)\frac{k\vert C\vert}s$ where $c(d,\delta)$ is the constant in the Benjamini-Schramm lemma.
\end{lem}

Recall that the {\em multiplicity at $x\in M$} of an open covering $M=\bigcup U_i$ is the number of members of the covering with $x\in U_i$. The {\em multiplicity} of the covering is the maximum of the multiplicities over all points in $M$.

\begin{proof}
Let $A_i$ be the set of $(\delta,s)$ separated points in $\phi_i(C\cap U_i)$. From the Benjamini-Schramm lemma we get that $\vert A_i\vert\le c(d,\delta)\frac{\vert C\cap U_i\vert}s$. It follows that the set of those $x\in C$ for which there is $i$ with $x\in U_i$ and such that $\phi_i(x)$ is a $(\delta,s)$-supported point in $\phi_i(C\cap U_i)$ has cardinality
$$\vert\bigcup_{i=1}^r\phi_i^{-1}(A_i)\vert\le\sum_{i=1}^r\vert A_i\vert\le c(d,\delta)\frac{\sum_i \vert C\cap U_i\vert}s\le c(d,\delta)\frac{k\vert C\vert}s$$
as we needed to prove.
\end{proof}

Lemma \ref{bs2} is only of any use if one can construct coverings in a controlled way. That is the case if $M$ is a surface of constant curvature. We denote by $g(M)$ the genus of $M$.

\begin{lem}\label{covering-surface}
There are $k$ and $\delta$ such that every orientable Riemannian surface $M$ with constant curvature $\kappa_M\equiv -1,0,1$ (and $\vol(M)=1$ if $M$ is a torus) has an open covering $M=U_1\cup\dots\cup U_r$ with the following properties:
\begin{itemize}
\item The cover has at most multiplicity $k$.
\item For every $x\in M$ there is $i$ such that $B(x,\delta,M)\subset U_i$.
\item Each $U_i$ admits a conformal embedding into $\BC$.
\end{itemize}
\end{lem}
\begin{proof}
If $\kappa_M\equiv 1$ then $M$ is the round sphere $\BS^2$ and we can consider the covering $\BS^2=U_1\cup U_2$ where $U_1,U_2$ are the complements of the north and south poles respectively. The statement then holds for $k=2$ and for any $\delta<\frac\pi 2$.

If $\kappa_M\equiv 0$ then $M$ is a flat torus with volume $1$. In particular there is a geodesic $\gamma\subset M$ of at most length $2$. The complement $U_1=M\setminus\gamma$ is a flat annulus. Notice that the central curve $\gamma'$ of $A$ is also a geodesic parallel to $\gamma$ in $M$ and that $\gamma$ and $\gamma'$ are at least at distance $\frac 14$. Set $U_2=M\setminus\gamma'$ and notice that the covering $M=U_1\cup U_2$ satisfies the claim for $k=2$ and for all $\delta<\frac 18$.

It remains to consider the case that $M$ is a hyperbolic surface. At this point we observe that there are two positive constants $\mu>\mu'$ such that the following holds:
\begin{itemize}
\item The $\mu$-thin part $M^{<\mu}=\{x\in M\vert\inj(M,x)<\mu\}$ has at most $3g(M)-3$ components and each one of them is homeomorphic to an annulus.
\item Every point $x\in M^{\ge\mu}=M\setminus M^{<\mu}$ in the $\mu$-thick part of $M$ is at least at distance $\mu'$ of $M^{<\mu'}$.
\end{itemize}

\begin{bem}
Concrete values for the constants $\mu$ and $\mu'$ can be computed using standard hyperbolic trigonometry but the reader who is only interested in their existence - as we are ourselves - can take $\mu$ to be the 2-dimensional Margulis constant \cite{Benedetti-Petronio} and take $\mu'=\frac 13\mu$.
\end{bem}

The cover of $M$ will consist of the components of $M^{<\mu}$ and of a finite collection of balls. To choose the centers of the balls take a maximal $\frac{\mu'}8$-separated set of points $x_1,\dots,x_s$ with $\inj(M,x_i)\ge\frac{\mu'}2$ for each $i$. The balls of radius $\frac{\mu'}4$ around these points are embedded and cover an open set containing $M^{\ge\mu'}$. In particular, the collection of annuli in $M^{<\mu}$ and of the balls $B(x_i,\frac{\mu'}4,M)$ are a covering and obviously each component is biholomorphic to a subset of $\BC$. 

Moreover, since the balls of radius $\frac{\mu'}{16}$ centered at $x_1,\dots,x_s$ are disjoint and each one of them has volume $\vol(B(0,\frac{\mu'}{16},\BH^2))$, it follows that each point in $M$ belongs to at most $\frac{\vol(B(0,\frac{\mu'}4,\BH^2)}{\vol(B(0,\frac{\mu'}{16},\BH^2)}$ balls. Since on the other hand each point is contained in at most one component of $M^{<\mu}$ it follows that the cover has multiplicity bounded by some universal constant $k$.

Finally, if $x$ is a point in $M^{<\mu}$ with $d(x,M^{\ge\mu})>\frac{\mu'}2$ then $B(x,\frac{\mu'}8,M)$ is contained in a connected component of $M^{<\mu}$. On the other hand, if $x$ is a point in $M^{\ge\mu'}$ with $d(x,M^{<\mu'})>\frac{\mu'}2$ then there is $x_i$ with $d_M(x,x_i)<\frac{\mu'}8$ and hence with $B(x,\frac{\mu'}8,M)\subset B(x_i,\frac{\mu'}4,M)$. This shows that the covering has at least Lebesgue number $\frac{\mu'}8$, as we needed to prove. This concludes the proof of Lemma \ref{covering-surface}.
\end{proof}

\section{Proof of Theorem \ref{sat1}}\label{sec:sat1}
As the reader surely suspects, we prove now Theorem \ref{sat1}. 

\begin{named}{Theorem \ref{sat1}}
Given $K\ge 1$ and a closed Riemannian manifold $M$ of dimension $d\ge 3$, let $(M_i)\subset\CQ(M,K)$ be a sequence with distributional limit $\lambda$. If $\vol(M_i)\to\infty$, then the set of those $(X,x)\in\CH$ such that $X$ is a Riemannian manifold $K$-quasi-conformally equivalent to $\BR^d$ or $\BR^d\setminus\{0\}$ has full $\lambda$-measure. 
\end{named}

For the convenience of the reader we recall the properties of the manifolds $M_i$ hidden by the notation. For each $i$, $M_i$ is a Riemannian manifold with pinched sectional curvature $\vert\kappa_{M_i}\vert\le 1$ and there is a $K$-quasi-conformal homeomorphism 
$$f_i:M_i\to M$$
which we consider fixed from now on. Recall that the injectivity radius of the manifolds $M_i$ is uniformly bounded from below by Theorem \ref{no-collapse}. Since our results are invariant under scaling each one of the manifolds $M_i$ by a bounded amount greater than one we can assume without loss of generality that 
$$\inj(M_i)\ge 10$$
for all $i$. 

Before going any further, fix for each $i$ a maximal 1-net $\CN_i\subset M$ and notice that the curvature and injectivity radius bounds imply that there is a constant $C\ge 1$ such that 
\begin{equation}\label{eq:volumeballs}
\frac 1C\le \vol_{M_i}(B(x,1/2,M_i)) < \vol_{M_i}(B(x,1,M_i))\le C
\end{equation}
for all $i$ and all $x\in M_i$. Since the balls of radius $1/2$ centered at points of $\CN_i$ are pairwise disjoint and $\CN_i$ is a maximal 1-net, we have
\begin{equation}\label{eq:volumenet}
\frac 1{C}\vol_{M_i}(M_i)\le\vert\CN_i\vert\le C\vol_{M_i}(M_i)
\end{equation}
In particular, the cardinalities of the nets $\CN_i$ tend to $\infty$ when $i$ grows.

\begin{lem}\label{lem:generic}
In the situation of Theorem \ref{sat1}, $\lambda$-almost every point $(X,x)\in\CH$ is such that $X$ is isometric to the Gromov-Hausdorff limit of a subsequence of a sequence $(M_i,x_i)$ with $x_i\in\CN_i$ satisfying:
\begin{enumerate}
\item For every $R>0$ we have $\lim_{i\to\infty}\diam_M(f_i(B(x_i,R,M_i)))=0$.
\item For every $\delta\in(0,1)$ there are $s$ and $i_s$ such that $f_i(x_i)$ is not $(\delta,s)$-supported in $f_i(\CN_i)\subset M$ for all $i\ge i_s$.
\end{enumerate}
\end{lem}

\begin{proof}
Choose sequences $\epsilon_i\to 0$, $r_i\to\infty$ and $s_n\to\infty$ satisfying 
$$\frac{e^{(d-1)r_i}r_i^d}{\epsilon_i^d\vert\CN_i\vert}\to 0\ \ \hbox{and}\ \ \frac{c(M,n^{-1})}{s_n}\to 0$$
where $c(M,n^{-1})$ is the constant provided by Lemma \ref{bs1}. Denote by $A_i(n)$ the subset of $\CN_i$ consisting of those points $p\in\CN_i$ for which one of the following is satisfied:
\begin{itemize}
\item $\diam(f_i(B(p,r_i,M_i)))\ge\epsilon_i$.
\item $f_i(p)$ is $(\frac 1n,s_n)$-supported in $f_i(\CN_i)$.
\end{itemize}
Consider also the set 
$$U_i(n)=M_i\setminus \bigcup_{p\in A_i(n)}B(p,1,M_i)$$ 
and observe that Lemma \ref{lem:generic} follows once we prove:
\medskip

\noindent{\bf Claim.} 
For $\lambda$-almost every $(X,x)\in\CH$ the following holds: For every $n_0$ there are $n\ge n_0$ and a sequence $(M_i,p_i)$ converging to $(X,x)$ with $p_i\in U_i(n)$ for all sufficiently large $i$. 
\medskip

To prove that this is the case we bound first the cardinality of $A_i(n)$. To begin with, Proposition \ref{prop:contraction} and the choice of $r_i$ and $\epsilon_i$ imply that
$$\lim_{i\to\infty}\frac{\vert\{p\in\CN_i\vert\diam_M(f_i(B(p,r_i,M_i)))\ge\epsilon_i\}\vert}{\vert\CN_i\vert}=0.$$
In particular for every $n$, we can choose $i_n$ so that for every $i\ge i_n$,
\[ \vert\{p\in\CN_i\vert\diam_M(f_i(B(p,r_i,M_i)))\ge\epsilon_i\}\vert \le \frac{c(M,n^{-1})}{s_n} \vert \CN_i \vert . \]
On the other hand, Lemma \ref{bs1} asserts that the set of all those $p\in\CN_i$ such that $f_i(p)$ is $(n^{-1},s_n)$-supported in $f_i(\CN_i)\subset M$ has at most $\frac{c(M,n^{-1})}{s_n}\vert\CN_i\vert$ elements. It follows that for all $i \ge i_n$ 
$$\vert A_i(n)\vert\le 2\frac{c(M,n^{-1})}{s_n}\vert\CN_i\vert\ \ \hbox{for all}\ \ i\ge i_n.$$
Now \eqref{eq:volumeballs} and \eqref{eq:volumenet} imply that for any such $i$
$$\vol_{M_i}(U_i(n))\ge\left(1-2C^2\frac{c(M,n^{-1})}{s_n}\right)\vol_{M_i}(M_i).$$
It follows hence from Lemma \ref{lem-generic} that the set $U_n$ of those $(X,x)\in\CH$ which are limits of subsequences of $(M_i,x_i)$ with $x_i\in U_i(n)$ satisfies
$$\lambda(U_n)\ge 1-2C^2\frac{c(M,n^{-1})}{s_n}$$
By the choice of $s_n$ we deduce that $\lambda(\bigcup_n U_n)=1$. The claim follows.
\end{proof}

After these preliminary considerations we can launch the proof of our main theorem:

\begin{proof}[Proof of Theorem \ref{sat1}]
By Lemma \ref{lem:generic} it suffices to prove that $X$ is quasi-conformal to $\BR^d$ or $\BR^d\setminus\{0\}$ if $(X,x)$ is a limit of a subsequence, say the whole sequence, of $(M_i,p_i)$ where $p_i\in\CN_i$ and satisfies:
\begin{enumerate}
\item For every $R>0$ we have $\diam_M(f_i(B(p_i,R,M_i)))\to 0$.
\item For every $\delta\in(0,1)$ there are $s$ and $i_s$ such that $f_i(p_i)$ is not $(\delta,s)$-supported in $f_i(\CN_i)\subset M$ for $i\ge i_s$.
\end{enumerate}
By Gromov's $C^{1,1}$-compactness theorem, the limit $X$ is a Riemannian manifold and the convergence takes place in the pointed $C^{1,\alpha}$-topology for all $\alpha$. In particular, there is an exhaustion of $X$ by nested open bounded connected subsets
$$\Omega_1\subset\bar\Omega_1\subset\Omega_2\subset\dots\subset X=\bigcup_{i=1}^\infty\Omega_i$$
and a sequence of $L_i$-bi-lipschitz embeddings 
$$\phi_i:(\Omega_i,x)\hookrightarrow(M_i,p_i)$$
with $L_i\to 1$.

For each fixed $j$ and all sufficiently large $i$ we consider the map 
$$f_i\circ\phi_i:(\Omega_j,x)\to(M,f_i(p_i)).$$
Fix a point $x_i\in\D B(x,1,X)$ so that 
$$d_M((f_i \circ \phi_i)(x),(f_i \circ \phi_i)(x_i)) = \max_{z\in B(x,1,X)} d_M((f_i \circ \phi_i)(p_i),(f_i \circ \phi_i)(z))=\colon r_i$$
and notice that $r_i\to 0$ by (1). Consider the scaled manifold $(\frac{1}{r_i}M,(f_i \circ \phi_i)(x))$ pointed at the image of the base point $x\in X$.

Since $r_i$ tends to $0$ we have that the maximum of the absolute value of the sectional curvature of the manifold $\frac{1}{r_i}M$ tends to $0$ and that its injectivity radius tends to $\infty$. It follows that $(\frac{1}{r_i}M,(f_i \circ \phi_i)(x))$ converges in the Gromov-Hausdorff topology to $(\BR^d,0)$. In other words, there are sequences $(L_i')$ and $(R_i)$ tending to $1$ and to $\infty$ respectively such that for all $i$ there exists a $L_i'$-bilipschitz embedding 
$$\varphi_i \colon B((f_i\circ\phi_i)(x),R_i,\frac{1}{r_i}M)\to(\BR^n,0).$$
Notice that by the choice of $r_i$, the diameter of $(f_i\circ\phi_i)(\Omega_j)$ is bounded in $\frac{1}{r_i}M$ independently of $i$. We may thus assume, passing to a subsequence if necessary, that
$$(f_i\circ\phi_i)(\Omega_j)\subset B((f_i\circ\phi_i)(x),R_i,\frac{1}{r_i}M)$$
for all $i$. Notice also that 
$$d_{\BR^n}(0,(\varphi_i\circ f_i\circ\phi_i)(x_i))=d_{\BR^n}((\varphi_i\circ f_i\circ\phi_i)(x),(\varphi_i\circ f_i\circ\phi_i)(x_i))\to 1$$
as $i\to\infty$. We can summarize the situation, for every fixed $j$ and all $i$ large enough, in the following diagram
$$\xymatrix{(M_i,p_i)\ar[r]^{f_i} & (M,f_i(p_i)) \ar[r]^{\Id} & (\frac 1{r_i} M_i,f_i(p_i))\\
(\phi_i(\Omega_j),p_i)\ar@{^{(}->}[u]\ar[r]^{f_i} & ((f_i\circ\phi_i)(\Omega_j),f_i(p_i)) \ar@{^{(}->}[u]\ar@{^{(}->}[r] & (B(f_i(p_i),R_i,\frac{1}{r_i}M),f_i(p_i))\ar[d]^{\varphi_i}\ar@{^{(}->}[u]\\
(\Omega_j,x)\ar[u]^{\phi_i}\ar@{^{(}->}[d]\ar[rr]^{\varphi_i\circ f_i\circ\phi_i} & & (\BR^d,0)\\
(X,x) & { } & { } }$$
The map $\varphi_i\circ f_i\circ\phi_i$ is the composition of a $K$-quasi-conformal map and two $L_i$- and $L_i'$-bilipschitz maps. Since $L_i,L_i'\to 1$ by construction, we have that $\phi_i\circ f_i\circ\psi_i$ is $(K+\epsilon)$-quasi-conformal for all $\epsilon>0$ and all sufficiently large $i$. Moreover, we have that $\phi_i\circ f_i\circ\psi_i$ maps the two points $x,x_i\in\Omega_1\subset\Omega_j$ with $d_X(x,x_i)=1$ to two points whose distance tends to $1$. It follows hence from Corollary \ref{cor:qc_limit} that, up to passing to a subsequence and possibly reducing the sets $\Omega_i$, the maps
$$\varphi_i\circ f_i\circ\phi_i:\Omega_j\to\BR^d$$
converge to a $K$-quasi-conformal embedding 
$$F:X=\bigcup_j\Omega_j\to\BR^d$$
To conclude the proof of Theorem \ref{sat1}, it remains to be shown that $\BR^d\setminus F(X)$ consists of at most a single point. Before launching the proof of this fact notice that, passing to a further subsequence, we obtain a maximal 1-net $\CN\subset X$ with $\CN\cap U$ equal to the Hausdorff limit of the sets $U\cap\phi_i^{-1}(\CN_i)$ for every $U\subset X$ open and bounded whose adherence does not meet $\CN$. 
\medskip

\noindent{\bf Claim.} {\em $F(\CN)$ has at most one accumulation point in $\BR^d$.}

\begin{proof}[Proof of the claim]
So far we have only used the fact that the base points $p_i$ satisfy condition (1) above. Now it comes the time to use that for every $\delta\in(0,1)$ there are $s$ and $i_s$ such that $f_i(p_i)$ is not $(\delta,s)$-supported in $f_i(\CN_i)\subset M$ for all $i\ge i_s$. Observe that being $(\delta,s)$-supported is invariant under scaling. In particular, we have that again for every $\delta\in(0,1)$ there are $s$ and $i_s$ such that $f_i(p_i)$ is not $(\delta,s)$-supported in $f_i(\CN_i)\subset\frac 1{r_i}M$ for all $i\ge i_s$. Since the bilipschitz constant of the maps $\varphi_i$ tends to $1$, this also implies that for every $\delta\in(0,1)$ there are $s$ and $i_s$ such that 
\begin{itemize}
\item[(*)] $0=f_i(p_i)$ is not $(\delta,s)$-supported in $(\varphi_i\circ f_i)(\CN_i)\subset\BR^d$ for $i\ge i_s$. 
\end{itemize}
We claim that this implies that $F(\CN)$ has at most one accumulation point but first we observe that the convergence of the maps $\varphi_i\circ f_i\circ\phi_i$ to the open map $F$ implies that there is some uniform $\epsilon$ with 
\begin{equation}\label{eq-unisep}
\rho_{\BR^d,(\varphi_i\circ f_i)(\CN_i)}(0)\ge\epsilon\ \ \hbox{for all}\ i.
\end{equation}
Arguing by contradiction suppose that there are at least two accumulation points $z,z'\in\BR^d$ of $F(\CN)$ and choose $\delta\in(0,1)$ so that 
\begin{equation}\label{eq;scales}
d_{\BR^n}(0,z),d_{\BR^n}(0,z')< \frac14 \delta^{-1}\epsilon,\ \ d_{\BR^n}(z,z')>4\delta\epsilon
\end{equation}
where $\epsilon$ satisfies \eqref{eq-unisep}. Since $z,z'$ are accumulation points of $F(\CN)$ and since every point in $F(\CN)$ is a limit of points in $(\varphi_i\circ f_i)(\CN_i)$ we can find for all $s$ some $i_s$ such that for all $i\ge i_s$ we have:
\begin{align}
\label{a2}&\vert (\varphi_i\circ f_i)(\CN_i)\cap B(z,\delta\epsilon,\BR^d)\vert> s\ \ \hbox{and}\\
\label{a3}&\vert (\varphi_i\circ f_i)(\CN_i)\cap B(z',\delta\epsilon,\BR^d)\vert> s
\end{align}
Now, \eqref{eq-unisep}, \eqref{eq;scales}, \eqref{a2} and \eqref{a3} show that $0$ is $(\delta,s)$-supported in $(\varphi_i\circ f_i)(\CN_i)\subset\BR^d$ for all $i$ large enough; compare with figure  \ref{fig1}. This contradiction to (*) shows that $F(\CN)$ has at most a single accumulation point, as we needed to prove.
\end{proof}

Having ruled out the possibility that the image $F(\CN)$ of the maximal 1-net $\CN$ under the quasi-conformal map $F$ has two accumulation points in $\BR^d$, we deduce from Lemma \ref{lemma:density} that $F(X)$ misses at most a point in  $\BR^d$, as we needed to prove. This concludes the proof of Theorem \ref{sat1}.
\end{proof}

\section{Riemannian parabolicity and the proof of Theorem \ref{cheeger}}\label{sec:cheeger}
A Riemannian manifold $X$ is {\em $p$-parabolic} for some $1<p<\infty$ if for all compact sets $E\subset X$ 
$$\mathrm{cap}_p(E,X) := \inf_{u} \int_X |\DD u|^p \vol = 0$$
where the infimum is taken over all compactly supported functions $u\in C^\infty_0(X)$ with $u(x)=1$ for $x\in E$. Equivalently
$$\Mod_p(\Gamma_\infty(X))=0$$ 
where $\Gamma_\infty(X)$ is the family of paths in $X$ leaving every compact set; see for example \cite[Theorem 5.4]{HP}. Having this characterization at our disposal it is easy to observe that we may equivalently say that $X$ is $p$-parabolic if there exists a continuum $E\subset X$ so that $\mathrm{cap}_p(E,X) = 0$. Indeed, let $E$ and $E'$ be continua in $X$ and suppose that $\mathrm{cap}_p(E,X)>0$. Let $\Gamma_\infty(X;E)$ and $\Gamma_\infty(X;E')$ be subfamilies of $\Gamma_\infty(X)$ consisting of paths meeting $E$ and $E'$, respectively. Since $
\Mod_p(\Gamma_\infty(X;E))=\mathrm{cap}_p(E,X)$, by the argument of \cite[Lemma 3.2]{MR} we have $\Mod_p(\Gamma_\infty(X;E'))>0$. Thus also $\mathrm{cap}_p(E';X)>0$.

Using either definition it is also easy to see directly that $d$-parabolicity is invariant under quasi-conformal homeomorphisms. Furthermore, the modulus inequality \eqref{eq:cap} readily yields that $\BR^d$ and $\BR^d\setminus\{0\}$ are $d$-parabolic.
 
In the light of all this, Theorem \ref{sat1} immediately implies:

\begin{kor}\label{cor:parabolicity}
Let $(M_i)$ be a sequence of Riemannian $d$-manifolds as in Theorem \ref{sat1}. The set of those $(X,x)\in \CH$ that are $d$-parabolic has $\lambda$-full measure.\qed
\end{kor}

Before deducing Theorem \ref{cheeger} from Corollary \ref{cor:parabolicity} we discuss briefly the following fact well-known to experts:

\begin{prop}\label{par-che}
Suppose that $X$ is a Riemannian manifold with $\vert \kappa_X\vert\le 1$ and with $\inj(X)>0$. If $X$ is parabolic then $X$ has vanishing Cheeger constant $h(X)=0$.
\end{prop}
\begin{proof}
Suppose that 
$$h(X) = \inf_{\Omega} \frac{\mathrm{Area}_X(\partial \Omega)}{\vol_X(\Omega)} >0$$
where the infimum is taken over all bounded domains $\Omega\subset X$. Noting that $\frac m{m-1}>1$ it follows that 
$$h(X) \vol_X(\Omega) \le \mathrm{Area}_X(\D\Omega)^{\frac {m}{m-1}} $$
for all $m> 1$ and all domains $\Omega\subset X$ of volume at least $\frac 1{h(X)}$. Since $\inj(X)>0$, this implies that $\mathrm{cap}_d(B, X)>0$ for every closed ball $B\subset X$ (see for example \cite[Lemma 2.7]{Holopainen}) implying that $X$ is not $d$-parabolic. This contradiction yields that $h(X)=0$.
\end{proof}

We are now ready to prove Theorem \ref{cheeger}:

\begin{named}{Theorem \ref{cheeger}}
Fix $K\ge 1$ and a closed Riemannian manifold $M$ of dimension $d\ge 3$. If $(M_i)\subset\CQ(M,K)$ is a sequence such that $\vol(M_i)\to\infty$, then 
$$\lim_{i\to\infty}h(M_i)= 0$$
where $h(M_i)$ is the Cheeger constant of $M_i$.
\end{named}
\begin{proof}
Seeking a contradiction, assume that there is a sequence $(M_i)\subset\CQ(M,K)$ with $\vol(M_i)\to \infty$ and $h(M_i)\ge\epsilon>0$ for all $i$. Passing to a subsequence we may assume that the sequence $(M_i)$ converges in distribution to some measure $\lambda$ on $\CH$. By Corollary \ref{cor:parabolicity} we know hence that there are base points $p_i\in M_i$ such that $(M_i,p_i)$ converges in the pointed Gromov-Hausdorff topology to some $(X,x)$ with $X$ $d$-parabolic. Notice that $\inj(X)>0$ by Theorem \ref{no-collapse} and hence that $h(X)=0$ by Proposition \ref{par-che}.

Since $h(X)=0$ there exists a bounded domain $\Omega\subset X$ so that 
$$\mathrm{Area}_X(\partial \Omega) \le \frac\epsilon 2 \vol_X(\Omega).$$
Let $\Omega'\subset X$ be a further bounded domain with $\Omega\subset\Omega'$. Since, by Gromov's $C^{1,1}$-compactness theorem, $(M_i,p_i)$ converges to $(X,x)$ in the $C^{1,\alpha}$-topology, we have for all but finitely many $i$ an $L_i$-bilipschitz embedding
$$\psi_i \colon \Omega' \to M_i$$
with $L_i\to 1$. We have hence 
$$h(M_i)\le\frac{\mathrm{Area}_{M_i}(\phi_i(\D\Omega))}{\vol_{M_i}(\phi_i(\Omega))}\le\frac{L_i^{d-1}\mathrm{Area}_X(\D\Omega))}{L_i^{-d}\vol_X(\Omega)}\le L_i^{2d-1}\frac\epsilon 2$$
This shows that $h(M_i)<\epsilon$ for all large $i$ contradicting our assumption. This proves Theorem \ref{cheeger}.
\end{proof}

\section{The 2-dimensional case}\label{sec:2dim}
In this section we consider the 2-dimensional case.

\begin{named}{Theorem \ref{sat12}}
Suppose that $(M_i)$ is a sequence of closed Riemannian surfaces with 
$$\vert\kappa_{M_i}\vert\le 1\ \hbox{and}\ \inj(M_i)>\epsilon>0$$ 
for all $i$. Suppose also that $(M_i)$ has distributional limit $\lambda$ and that 
$$\lim_{i\to\infty}\frac{g(M_i)+1}{\vol(M_i)}= 0$$
where $g(M_i)$ is the genus of $M_i$. Then $\lambda$ is supported by the set of Riemannian surfaces conformally equivalent to $\BC$ or $\BC^*$.
\end{named}
\begin{proof}
By assumption the surfaces $M_i$ have injectivity radius uniformly bounded from below. In particular we can, as in the proof of Theorem \ref{sat1}, scale them by a uniform amount and assume that $\inj(M_i)\ge 10$ for all $i$. Again as in the proof of Theorem \ref{sat1} we choose a maximal 1-net $\CN_i\subset M_i$ for all $i$. We also choose for each $i$ a uniformization $f_i:M_i\to\Sigma_i$ by what we mean that $f_i$ is conformal and $\Sigma_i$ is a Riemannian surface with sectional curvature $\kappa_{\Sigma_i}\equiv -1,0,1$.

As so often, the cases of $\kappa_{\Sigma_i}\equiv 0$ and $\kappa_{\Sigma_i}\equiv 1$, i.e. $M_i$ a torus and a sphere respectively are slightly particular but easier. We leave them to the reader and, for the sake of concreteness, assume from now on that $\Sigma_i$ is a hyperbolic surface for all $i$.

For each $i$ we fix a uniform atlas of $\Sigma_i$ as provided by Lemma \ref{covering-surface}, i.e.~an open covering 
$$\Sigma_i=U_i^1\cup\dots\cup U_i^{r_i}$$
with multiplicity $\le k$, such that for every $x\in\Sigma_i$ there is $j$ such that $B(x,\delta_0,\Sigma_i)\subset U_i^j$, and such that for all $i$ and $j$ there is a conformal embedding 
$$\varphi_i^j:U_i^j\hookrightarrow\BC$$
We stress the fact that $k$ and $\delta_0$ are independent of $i$. For example, the uniform bound on the multiplicity and Lemma \ref{bs2} imply that 
\begin{itemize}
\item[(*)] for every $\delta\in(0,1)$, every $s$ and every $i$ the set of those $x\in f_i(\CN_i)$ for which there is $j$ with $x\in U_i^j$ and such that $\phi_i^j(x)$ is $(\delta,s)$-supported in $\phi_i^j(f_i(\CN_i)\cap U_i^j)$ has cardinality at most $c(d,\delta)\frac{k\vert\CN_i\vert}s$ for $d=2$.
\end{itemize}

On the other hand, the closed hyperbolic surface $\Sigma_i$ has genus $g(M_i)$ and hence volume $4\pi(g(M_i)-1)$. In particular, the assumption in Theorem \ref{sat12} implies that
$$\lim_{i\to\infty}\frac{\vol(\Sigma_i)}{\vol(M_i)}= 0$$
Hence, Proposition \ref{prop:contraction} implies that there are $C$ and $\epsilon_0$ such that
\begin{itemize}
\item[(**)] for all $i$, $R$ and $\epsilon<\epsilon_0$ we have
$$\vert\{p\in\CN_i\vert\diam_{\Sigma_i}(f_i(B(p,R,M_i)))\ge\epsilon\}\vert\le C\vert\CN_i\vert \left( \frac R{\epsilon}\right)e^{R}$$
\end{itemize}
where we are also using the fact that $\vol(\Sigma_i) \le \vol(M_i) \le C' \vert \CN_i \vert$ for some uniform constant $C'$ and $i$ large.
Armed with (*) and (**) we can repeat word-by-word the proof of Lemma \ref{lem:generic} and show that $\lambda$-almost every point $(X,x)\in\CH$ is such that $X$ is isometric to the Gromov-Hausdorff limit of subsequence of a sequence $(M_i,x_i)$ with $x_i\in\CN_i$ satisfying:
\begin{enumerate}
\item For every $R>0$ we have $\lim_{i\to\infty}\diam_{\Sigma_i}(f_i(B(x_i,R,M_i)))=0$.
\item For every $\delta\in(0,1)$ there are $s$ and $i_s$ such that for all $i\ge i_s$ the following holds: if $f_i(x_i)\in U_i^j$ then $\phi_i^j(f_i(x_i))$ is not $(\delta,s)$-supported in $\phi_i^j(f_i(\CN_i)\cap U_i^j)\subset\BC$.
\end{enumerate}
We proceed now as in the proof of Theorem \ref{sat1}. To begin with, Gromov's $C^{1,1}$-compactness theorem implies that the limit $X$ is a Riemannian manifold and that the convergence takes place in the pointed $C^{1,\alpha}$-topology. As when proving Theorem \ref{sat1} we obtain, up to passing to subsequences, an exhaustion 
$$\Omega_1\subset\bar\Omega_1\subset\Omega_2\subset\dots\subset X=\bigcup_{i=1}^\infty\Omega_i$$
of $X$ by bounded open connected sets and for all $j$ a sequence of $L_i$-bi-lipschitz embeddings 
$$\phi_i:(\Omega_j,x)\hookrightarrow(M_i,p_i)$$
with $L_i\to 1$. Since the diameter of $\phi_i(\Omega_j)$ remains bounded we deduce from (1) that 
$$\lim_{i\to\infty}\diam_{\Sigma_i}((f_i\circ\phi_i)(\Omega_j))=0.$$
Since the sequence of coverings $\Sigma_i=\bigcup_jU_i^j$ have uniform Lebesgue number $\delta_0$, it follows that for all sufficiently large $i$ there is some member of the covering, say $U_i^1$, with $(f_i\circ\phi_i)(\Omega_j)\subset U_i^1$. The composition
$$(\varphi_i^1\circ f_i\circ\phi_i):\Omega_j\to\BC$$
of $f_i\circ\phi_i$ with the conformal embedding $\varphi_i^1:U_i^1\to\BC$ is the composition of an $L_i$-bilipschitz map with two conformal maps and hence is $L_i^2$-quasi-conformal. Composing $\varphi_i^1$ with a translation and a homothety we can assume that $\varphi_i^1\circ f_i\circ\phi_i$ maps always $x$ to $0$ and a fixed point $y\in\D B(x,1,X)$ to $1$. As in the proof of Theorem \ref{sat1}, Corollary \ref{cor:qc_limit} implies that up to passing to a subsequence the maps $\varphi_i\circ f_i\circ\phi_i$ converge to a $1$-quasi-conformal, i.e. conformal, embedding 
$$F:X\to\BC$$
Again as in the proof of Theorem \ref{sat1}, we may assume that the nets $\CN_i\subset M_i$ converge to a 1-net $\CN$ of $X$ and again it suffices to prove that $F(\CN)$ has at most an accumulation point in $\BC$. The same argument, in fact a bit easier, as in the proof of Theorem \ref{sat1} shows that if that were not the case, then there would exist $\delta\in(0,1)$ such that for all $s$ there is $i_s$ such that the point $(\varphi_i^1\circ f_i\circ\phi_i)(x)=(\varphi^1_i\circ f_i)(p_i)$ is $(\delta,s)$-supported in $(\varphi^1_i\circ f_i)(\CN_i)$ for all $i\ge i_s$. This would contradict (2). We have proved Theorem \ref{sat12}.
\end{proof}

\begin{bem}
For later use we observe that in the course of the proof of Theorem \ref{genusgbs} we in fact obtained that if $\CN_i$ is a maximal 1-net of $M_i$ then there is a sequence $(U_i)$ with $U_i\subset\CN_i$, with $\lim_{i\to\infty}\frac{\vert U_i\vert}{\vert\CN_i\vert}=1$ and such that any Gromov-Hausdorff limit $(X,x)$ of a subsequence of $(M_i,x_i)$ with $x_i\in U_i$ is such that $X$ is conformally equivalent to either $\BC$ or $\BC^*$.
\end{bem}

\noindent{\bf An example showing that Theorem \ref{sat12} fails in the absence of injectivity radius bounds.} We claim that {\em for all $\delta<1$ there is a sequence $M_i$ of Riemannian surfaces with curvature $\vert\kappa_{M_i}\vert\le 1$, homeomorphic to $\BS^2$, with $\vol(M_i)\to\infty$,  and with distributional limit satisfying $$\lambda\left(\{(X,x)\in\CH\vert\ \pi_1(X)\ \hbox{is not finitely generated}\}\right)>\delta$$}
To construct the desired sequence we proceed as in the example in section \ref{sec:weak}. Let $T$ be a trivalent tree, $t\in T$ be a vertex and denote by $T_i$ the ball in $T$ centered at $t$ and with radius $i$. For all $i$ there is a hyperbolic surface $N_i$ with totally geodesic boundary which has a pants decomposition with dual graph $T_i$ and such that all the involved interior curves corresponding to edges of $T_i$ have length $1$ and such that every boundary component has length $\frac 1i$. To get the surface $M_i$, we cap off each one of the boundary components by a suitable bubble with curvature in $[-1,1]$ and of volume less than $100$. Here, the word suitable means that for example there is a function $f:[0,1]\to[0,1]$ with $f(0)=0$ such that for each $i$ we have
$$\vol(\{x\in M_i\vert\inj_{M_i}(x)<\epsilon\})<f(\epsilon)\vol(M_i).$$
It follows that for any weak limit $\lambda$ of any subsequence of $(M_i)$ we have
$$\lambda\left(\{(X,x)\in\CH\vert\ X\ \hbox{is a manifold of dimension}\ 2\}\right)=1.$$
On the other hand, for any sequence of points $x_i\in N_i\subset M_i$ such that $(M_i,x_i)$ converges to a surface $(X,x)$ we have that $\pi_1(X)$ is not finitely generated. Since there is some $\delta_0>0$ such that $\vol(N_i)>\delta_0\vol(M_i)$, it follows that
$$\lambda\left(\{(X,x)\in\CH\vert\ \pi_1(X)\ \hbox{is not finitely generated}\}\right)>\delta_0.$$
In order to construct such sequences for arbitrary $\delta$ we endow the surface $N_i$ with a metric of constant curvature $-\epsilon$ where $\epsilon$ is a sufficiently small positive number. We leave the details to the reader.

\section{Parabolicity on graphs and sequences with sublinear genus growth}\label{sec:graphs}
In this section we prove Theorem \ref{genusgbs} and Corollary \ref{expanders}, but first we have to clarify what we mean by distributional limits of graphs: so far we have only considered such limits for Riemannian manifolds. We refer to \cite{Woess} for background results and definitions concerning random walks on graphs.
\medskip

Suppose that $(G_i)$ is a sequence of finite graphs with uniformly bounded valence. Denote by $V(G_i)$ the set of vertices of $G_i$. We endow $G_i$ with the interior distance with respect to which each edge has length $1$. For each $i$ we consider the map
$$V(G_i)\to\CH,\ \ v\mapsto(G_i,v)$$
where, as all along, $\CH$ is the space of all pointed metric spaces with respect to the Gromov-Hausdorff topology. On $V(G_i)$ we have the probability measure which gives equal weight to each vertex. Pushing forward these measures we get a sequence $(\lambda_i)$ of measures on $\CH$. We say that a measure $\lambda$ is the {\em distributional limit} of $(G_i)$ if the sequence $(\lambda_i)$ converges to $\lambda$ in the weak-*-topology. See \cite{BS} for details.

\subsection{Parabolicity}
Suppose now that $G$ is an infinite graph and let $E(G)$ be the set of its edges. We say that $G$ is \emph{$p$-parabolic} if for every non-empty finite set $S\subset V(G)$ 
\begin{equation}
\label{eq:para_graph}
\mathrm{cap}_p(S,G) := \inf_u \sum_{q\in E(G)} |\DD u(q)|^p = 0,
\end{equation}
where the infimum is taken over all finitely supported function $u\colon G\to \BR$ so that $u\ge 1$ on $S$ and
\[ | \DD u(q) | = \left( \sum_{(q,q')\in E(G)} (u(q') - u(q))^2 \right)^{1/2}; \]
we refer to \cite[Section 5]{Holopainen} for the definition and terminology. 

As in the case of Riemannian manifolds, we may introduce the $p$-modulus $\Mod_p(\Gamma)$ of a family $\Gamma$ of self-avoiding paths on $G$:
\[
\Mod_p(\Gamma) = \inf_\rho \sum_{v\in G} \rho(v)^p
\]
where $\rho$ is a non-negative function on $G$ so that 
\[
\sum_{v\in \gamma} \rho(v) \ge 1
\]
for every $\gamma\in \Gamma$.

It is now easy to prove \cite{Woess} that $\mathrm{cap}_p(S,G) = \Mod_{p}(\Gamma(S))$, where $\Gamma(S)$ is the set of self avoiding infinite paths in $G$ meeting $S$. 

\begin{bem}
Since $d-\mathrm{VEL}(\Gamma(v)) = \Mod_p(\Gamma(v))^{-1}$, where $p-\mathrm{VEL}(\Gamma(v))$ is  the \emph{$p$-vertex extremal length of $\Gamma(v)$} in \cite{BC}, we have that the definition of $p$-parabolicity in \cite{BC} coincides with the definition above.
\end{bem}

The following fact is immediate from either the modulus or the capacity definitions.
\medskip

\noindent{\bf Fact.} Every subgraph of a $p$-parabolic graph is $p$-parabolic.
\medskip

Notice at this point that the same argument as we used in the proof of Theorem \ref{cheeger} shows that $p$-parabolicity for some $p$ of an infinite graph $G$ of bounded valence implies that the Cheeger constant $h(G)$ vanishes (compare also with the argument in \cite[Corollary 4.2]{BC}). Recall that the Cheeger constant of an infinite graph $G$ is the infimum of $\frac{\vert\D A\vert}{\vert A\vert}$ over of finite sets of vertices $A\subset V(G)$ where $\D A\subset E(G)$ is the set of edges separating $A$ from its complement $V(G)\setminus A$. 

It also well-known that 2-parabolicity of a graph $G$ of bounded valence is equivalent to the recurrence of the simple random walk on $G$ \cite{Woess}. We record these facts for later use:
\medskip

\noindent{\bf Fact.} Let $G$ be an infinite graph of bounded valence. If $G$ is $p$-parabolic for some $p$ then $h(G)=0$. If $G$ is 2-parabolic then the simple random walk on $G$ is recurrent.
\medskip

We are now ready to record the invariance of $p$-parabolicity under quasi-isometries. Recall that a map $f:Y\to Z$ between metric spaces is a \emph{quasi-isometry} if there exists $c>0$ so that 
$$c^{-1} d_Y(x,y) -c \le d_Z(f(x),f(y)) \le c d_Y(x,y) +c$$
for all $x,y \in Y$. Two metric space are said to be {\em quasi-isometric} if there is a quasi-isometry whose image is $r$-dense for some $r<\infty$.
 
\begin{sat*}[Kanai-Holopainen]
Let $G$ be a graph of uniformly bounded valence which is quasi-isometric to a Riemannian manifold $M$ with bounded sectional curvature and positive injectivity radius. Then $G$ is $p$-parabolic if and only if $M$ is $p$-parabolic.
\end{sat*}

Kanai proved this theorem in \cite{Kanai} for $p=2$; the extension to $p\neq 2$ is due to Holopainen \cite[Lemma 5.9]{Holopainen}. It should be noticed that these authors prove a more general form of the above theorem: instead of a bound on the sectional curvature, they merely assume that $M$ has Ricci curvature bounded from below.

\subsection{From graphs to metrics}
In this section we adapt an argument due to Benjamini and Schramm allowing to extend  graphs in surface to triangulations in a controlled way (compare with the proof of \cite[Theorem 1.1]{BS}). Then, we describe how to associate a Riemannian metric to every triangulation.

\begin{lem}\label{fill}
Let $G$ be a graph of valence $d$ embedded in a closed surface $M$ and assume that $g(G)=g(M)$. There is a triangulation $T$ of $M$ containing $G$ and satisfying:
\begin{itemize}
\item Every vertex of $T$ is already a vertex of $G$.
\item $T$ has at most valence $6d$.
\end{itemize}
\end{lem}

\begin{proof}
Every component of $M\setminus G$ is either a disk or an annulus because $g(G)=g(M)$. For every annulus in $M\setminus G$ take an arc with endpoints in vertices of $G$ and add it to $G$. Let $G'$ be the so obtained graph and notice that because every vertex of $G$ is in the closure of at most $d$ components of $M\setminus G$ the graph $G'$ has at most valence $2d$. Now, every component of $M\setminus G'$ is a disk. Every polygon can be triangulated with at most valence $4$ and without adding vertices. When doing this for every component of $M\setminus G'$ we obtain a triangulation $T$ of $M$ with no new vertices, containing $G'$ and hence $G$, and with at most valence $6d$.
\end{proof}

Suppose now that $T$ is a triangulation of at most valence $d$ of a surface $M$. Identifying every face of $T$, i.e. every component of $M\setminus T$, with an euclidean equilateral triangle of side-length $1$ we obtain a piecewise euclidean metric on $M$ which is in fact Riemannian outside of the vertices of $T$. Notice that the balls of radius $\frac 13$ around any two vertices of $T$ are disjoint and that since $T$ has valence $\le d$ there are at most $d$ isometry classes of such balls. Choose once and for ever a procedure to smooth out the singularity for each one of those $d$ models and apply the corresponding procedure to every ball of radius $\frac 13$ around a vertex of $T$. When doing so one obtains a smooth Riemannian metric on $M$ which has curvature pinched by $\kappa(d)$, injectivity radius bounded from below by $\epsilon(d)$, and which is $c(d)$ quasi-isometric to $T$. Scaling by a bounded amount we have:

\begin{lem}\label{trian-metric}
For every $d$ there is $L$ such that for every triangulations $T$ with valence $\le d$ of a surface $M$ there is a Riemannian metric on $M$ such that:
\begin{itemize}
\item $\vert\kappa_M\vert\le 1$ and $\inj(M)\ge 1$.
\item There is an $L$-bilipschitz embedding $T\hookrightarrow M$ with $L$-dense image.\qed
\end{itemize}
\end{lem}

A similar construction of a metric associated to a triangulation was used by Gill-Rohde in \cite{Gill-Rohde}.

\subsection{Sequences of graphs with sublinear genus growth}
We are now ready to prove Theorem \ref{genusgbs} and Corollary \ref{expanders}:

\begin{named}{Theorem \ref{genusgbs}}
Let $(G_i)$ be a sequence of graphs with uniformly bounded valence, with $\vert G_i\vert\to\infty$, and with sublinear genus growth. If $(G_i)$ converges in distribution to $\lambda$, then $\lambda$ is supported by the set of 2-parabolic rooted graphs $(G,p)$. In particular, $\lambda$ is supported by rooted graphs with vanishing Cheeger constant $h(G)=0$ and recurrent simple random walk.
%
%
%Let $(G_i)$ be a sequence of graphs with uniformly bounded valence, with $\vert G_i\vert\to\infty$, and with sublinear genus growth. If $(G_i)$ converges in distribution to $\lambda$, then $\lambda$ is supported by the set of rooted graphs $(G,p)$ satisfying:
%\begin{itemize}
%\item $G$ is 2-parabolic,
%\item $G$ has vanishing Cheeger constant $h(G)=0$, and
%\item the simple random walk on $G$ is recurrent.
%\end{itemize}
\end{named}

\begin{proof}
For each $i$ consider a surface $M_i$ with genus $g(M_i)=g(G_i)$ and fix an embedding $G_i\hookrightarrow M_i$. Let also $T_i$ be the triangulations provided by Lemma \ref{fill} and notice that since the graphs $G_i$ have uniformly bounded valence, the same is true for the triangulations $T_i$.  Recall also the $T_i$ and $G_i$ have the same set of vertices $V(G_i)=V(T_i)$. For each $i$ let $\rho_i$ be the Riemannian metric provided by Lemma \ref{trian-metric} and write from now on $M_i=(M_i,\rho_i)$. 
\medskip

\noindent{\bf Claim.} There are subsets $V_i\subset V(G_i)$ with $\lim_{i\to\infty}\frac{\vert V_i\vert}{\vert G_i\vert}=1$ and such every Gromov-Hausdorff limit $(X,x)$ of a subsequence of a sequence $(M_i,p_i)$ with $p_i\in V_i$ is such that $X$ is conformally equivalent to $\BC$ or to $\BC^*$. 

\begin{proof}[Proof of the Claim]
Let $L$ be the constant in Lemma \ref{trian-metric} and observe that the subset $V(G_i)=V(T_i)$ of $M_i$ is $\frac 1L$-separated and $2L$-dense. In particular, there is some $c>0$ independent of $i$ with 
$$c\vert G_i\vert\ge \vol_{M_i}(M_i)\ge c^{-1}\vert G_i\vert$$
implying that the sequence $(M_i)$ satisfies the conditions in Theorem \ref{genusgbs}. Choose for all $i$ a 1-net $\CN_i$ and recall that there is $c'>0$ with 
$$c'\vert\CN_i\vert\ge \vol_{M_i}(M_i)\ge (c')^{-1}\vert\CN_i\vert.$$
As remarked after the proof of Theorem \ref{sat12} there is a sequence $(U_i)$ of subsets $U_i\subset\CN_i$, with $\lim_{i\to\infty}\frac{\vert U_i\vert}{\vert\CN_i\vert}=1$ and such that any Gromov-Hausdorff limit $(X,x)$ of a subsequence of $(M_i,x_i)$ with $x_i\in U_i$ is such that $X$ is conformally equivalent to either $\BC$ or $\BC^*$. This implies that if $(p_i)$ is any sequence with $p_i\in V(G_i)$ such that $d_{M_i}(p_i,U_i)\le 1$, then again any Gromov-Hausdorff limit $(X,x)$ of a subsequence of $(M_i,p_i)$ is conformally equivalent to either $\BC$ or $\BC^*$; set 
$$V_i=\{p\in V(G_i)\vert d_{M_i}(p,U_i)\le 1\}.$$
Notice now that there is some $c''>0$ depending only on $L$ and the fact that $\vert\kappa_{M_i}\vert\le 1$ asserting that every ball $B(x,1,M_i)$ contains at most $c''$ elements of the $\frac 1{L}$-separated set $V(G_i)$. It follows that
$$\vert V(G_i)\setminus V_i\vert\le c''\vert\CN_i\setminus U_i\vert$$
This shows that 
$$\frac{\vert V(G_i)\setminus V_i\vert}{\vert G_i\vert}\le cc'c''\frac{\vert\CN_i\setminus U_i\vert}{\vert\CN_i\vert}\to 0$$
as $i\to\infty$. This concludes the proof of the claim.
\end{proof}

Notice now that the claim proves that the distributional limit $\lambda$ of the sequence $(G_i)$ is supported by graphs $(G,p)$ which arise as Gromov-Hausdorff limits of sequences $(G_i,p_i)$ with $p_i\in V_i$. We show that any such $G$ satisfies the claims of Theorem \ref{genusgbs}. To begin with observe that up to passing to a subsequence we may assume that both the pointed triangulations $(T_i,p_i)$ and the pointed surfaces $(M_i,p_i)$ converge in the Gromov-Hausdorff topology to $(T,q)$ and $(X,x)$ respectively; by the claim $X$ is conformally equivalent to either $\BC$ or $\BC^*$. We can then pass again to a subsequence an assume that the embeddings $(G_i,p_i)\hookrightarrow(T_i,p_i)$ converge to an embedding $(G,p)\hookrightarrow(T,q)$. Similarly, we can assume that the $L$-bilipschitz embeddings $(T_i,p_i)\hookrightarrow(M_i,p_i)$ converge to an $L$-bilipschitz embedding $(T,q)\hookrightarrow(X,x)$. Moreover, since each one of the $T_i$ is $L$-dense in $M_i$ we obtain that $T$ is also $L$-dense in $X$. It follows that $T$ and $X$ are quasi-isometric.

Since $X$ is conformally equivalent to either $\BC$ or to $\BC^*$, it is 2-parabolic. In particular, the Kanai-Holopainen theorem implies that $T$ is also 2-parabolic. This implies that $G$ is also 2-parabolic because as we remarked above every subgraph of a 2-parabolic graph is 2-parabolic.  Once we know that $G$ is 2-parabolic, it follows that $h(G)=0$ and that the simple random walk on $G$ is recurrent. We have proved Theorem \ref{genusgbs}.
\end{proof}

All that remains to be done is to prove Corollary \ref{expanders}. This is done in very same way as to prove Theorem \ref{cheeger}.

\begin{named}{Corollary \ref{expanders}}
For every expander $(G_i)$ there is a positive constant $c>0$ with $g(G_i)\ge c\vert G_i\vert$ for all $i$.
\end{named}
\begin{proof}
Recall that an {\em expander} is a sequence $(G_i)$ of graphs with uniformly bounded valence such that $\vert G_i\vert\to\infty$ but such that there is $\epsilon$ positive with $h(G_i)>\epsilon$ for all $i$. Seeking a contradiction suppose that the sequence $(G_i)$ has sublinear genus growth. Up to passing to a subsequence we may assume that $(G_i)$ has some distributional limit $\lambda$. It follows hence from Theorem \ref{genusgbs} that there are base points $p_i\in G_i$ such that $(G_i,p_i)$ converges in the Gromov-Hausdorff topology to a graph $(G,p)$ with $h(G)=0$. This implies that there is $A\subset V(G)$ finite with 
$$\vert\D A\vert\le\frac\epsilon 2\vert A\vert$$
The Gromov-Hausdorff convergence of $(G_i,p_i)\to(G,p)$ implies that for all large $i$ there is an embedding $A\hookrightarrow G_i$ such that the sets of those points in $G$ and $G_i$ which are at distance $1$ from $A$ are isomorphic as graphs. In particular we obtain that $h(G_i)\le\frac\epsilon2$ for all $i$ large enough. This contradiction concludes the proof of Corollary \ref{expanders}.
\end{proof}

\bigskip

\noindent Department of Mathematics, University of Texas, Austin.
\newline \noindent
\texttt{hossein@math.utexas.edu}

\bigskip

\noindent Matematiikan ja tilastotieteen laitos, Helsingin yliopisto, Helsinki \emph{and} \newline Matematiikan ja tilastotieteen laitos, Jyv\"askyl\"an yliopisto, Jyv\"askyl\"a.
\newline \noindent
\texttt{pekka.pankka@jyu.fi}

\bigskip

\noindent Department of Mathematics, University of British Columbia, Vancouver.
\newline \noindent
\texttt{jsouto@math.ubc.ca}


\begin{thebibliography}{10}

\bibitem{samurais}
M.~{Abert}, N.~{Bergeron}, I.~{Biringer}, T.~{Gelander}, N.~{Nikolov},
  J.~{Raimbault}, and I.~{Samet}.
\newblock {On the growth of Betti numbers of locally symmetric spaces}.
\newblock {\em ArXiv e-prints}, Apr. 2011.

\bibitem{Benedetti-Petronio}
R.~Benedetti and C.~Petronio.
\newblock {\em Lectures on hyperbolic geometry}.
\newblock Universitext. Springer-Verlag, Berlin, 1992.

\bibitem{BC}
I.~Benjamini and N.~Curien.
\newblock On limits of graphs sphere packed in {E}uclidean space and
  applications.
\newblock {\em European J. Combin.}, 32(7):975--984, 2011.

\bibitem{BS}
I.~Benjamini and O.~Schramm.
\newblock Recurrence of distributional limits of finite planar graphs.
\newblock {\em Electron. J. Probab.}, 6:no. 23, 13 pp. (electronic), 2001.

\bibitem{Buser}
P.~Buser.
\newblock A note on the isoperimetric constant.
\newblock {\em Ann. Sci. {\'E}cole Norm. Sup. (4)}, 15(2):213--230, 1982.

\bibitem{doCarmo}
M.~P. {do Carmo}.
\newblock {\em Riemannian geometry}.
\newblock Mathematics: Theory \& Applications. Birkh{\"a}user Boston Inc.,
  Boston, MA, 1992.
\newblock Translated from the second Portuguese edition by Francis Flaherty.

\bibitem{sep2}
J.~R. Gilbert, J.~P. Hutchinson, and R.~E. Tarjan.
\newblock A separator theorem for graphs of bounded genus.
\newblock {\em J. Algorithms}, 5(3):391--407, 1984.

\bibitem{Gill-Rohde}
J.~T. {Gill} and S.~{Rohde}.
\newblock {On the Riemann surface type of Random Planar Maps}.
\newblock {\em ArXiv e-prints}, Jan. 2011.

\bibitem{Gromov-negcur}
M.~Gromov.
\newblock Manifolds of negative curvature.
\newblock {\em J. Differential Geom.}, 13(2):223--230, 1978.

\bibitem{Gromov-book}
M.~Gromov.
\newblock {\em Metric structures for {R}iemannian and non-{R}iemannian spaces},
  volume 152 of {\em Progress in Mathematics}.
\newblock Birkh{\"a}user Boston Inc., Boston, MA, 1999.

\bibitem{Heinonen-Book}
J.~Heinonen.
\newblock {\em Lectures on analysis on metric spaces}.
\newblock Universitext. Springer-Verlag, New York, 2001.

\bibitem{HK}
J.~Heinonen and P.~Koskela.
\newblock Quasiconformal maps in metric spaces with controlled geometry.
\newblock {\em Acta Math.}, 181(1):1--61, 1998.

\bibitem{Holopainen}
I.~Holopainen.
\newblock Rough isometries and {$p$}-harmonic functions with finite {D}irichlet
  integral.
\newblock {\em Rev. Mat. Iberoamericana}, 10(1):143--176, 1994.

\bibitem{HP}
I.~Holopainen and P.~Pankka.
\newblock {$p$}-{L}aplace operator, quasiregular mappings and {P}icard-type
  theorems.
\newblock In {\em Quasiconformal mappings and their applications}, pages
  117--150. Narosa, New Delhi, 2007.

\bibitem{Kanai}
M.~Kanai.
\newblock Rough isometries and the parabolicity of {R}iemannian manifolds.
\newblock {\em J. Math. Soc. Japan}, 38(2):227--238, 1986.

\bibitem{Kuiper}
N.~H. Kuiper.
\newblock On conformally-flat spaces in the large.
\newblock {\em Ann. of Math. (2)}, 50:916--924, 1949.

\bibitem{MR}
P.~Mattila and S.~Rickman.
\newblock Averages of the counting function of a quasiregular mappping.
\newblock {\em Acta Math.}, 143(3-4):273--305, 1979.

\bibitem{Petersen-art}
P.~Petersen.
\newblock Convergence theorems in {R}iemannian geometry.
\newblock In {\em Comparison geometry ({B}erkeley, {CA}, 1993--94)}, volume~30
  of {\em Math. Sci. Res. Inst. Publ.}, pages 167--202. Cambridge Univ. Press,
  Cambridge, 1997.

\bibitem{Petersen-book}
P.~Petersen.
\newblock {\em Riemannian geometry}, volume 171 of {\em Graduate Texts in
  Mathematics}.
\newblock Springer-Verlag, New York, 1998.

\bibitem{Rickman-book}
S.~Rickman.
\newblock {\em Quasiregular mappings}, volume~26 of {\em Ergebnisse der
  Mathematik und ihrer Grenzgebiete (3) [Results in Mathematics and Related
  Areas (3)]}.
\newblock Springer-Verlag, Berlin, 1993.

\bibitem{Vaisala}
J.~V{\"a}is{\"a}l{\"a}.
\newblock {\em Lectures on {$n$}-dimensional quasiconformal mappings}.
\newblock Springer-Verlag, Berlin, 1971.
\newblock Lecture Notes in Mathematics, Vol. 229.

\bibitem{Woess}
W.~Woess.
\newblock {\em Random walks on infinite graphs and groups}, volume 138 of {\em
  Cambridge Tracts in Mathematics}.
\newblock Cambridge University Press, Cambridge, 2000.

\bibitem{Z1}
V.~A. Zorich.
\newblock The theorem of {M}. {A}. {L}avrent'ev on quasiconformal mappings in
  space.
\newblock {\em Mat. Sb.}, 74:417--433, 1967.

\bibitem{Z2}
V.~A. Zorich.
\newblock Quasiconformal immersions of {R}iemannian manifolds, and a
  {P}icard-type theorem.
\newblock {\em Funktsional. Anal. i Prilozhen.}, 34(3):37--48, 96, 2000.

\end{thebibliography}
\end{document}